\documentclass[draft,11pt,english]{article}

\usepackage{amsfonts,amsmath,amsthm,amscd,amssymb,latexsym,verbatim,
caption2}

\theoremstyle{plain}
\newtheorem{theorem}{Theorem}[section]
\newtheorem{lemma}{Lemma}[section]
\newtheorem{proposition}{Proposition}[section]
\newtheorem{corollary}{Corollary}[section]
\newtheorem{definition}{Definition}[section]
\theoremstyle{definition}
\newtheorem{example}{Example}[section]
\newtheorem{remark}{Remark}[section]
\newcommand{\keywords}{\textbf{Key words. }\medskip}
\newcommand{\subjclass}{\textbf{MSC 2010. }\medskip}
\renewcommand{\abstract}{\textbf{Abstract. }\medskip}

\numberwithin{equation}{section}

\begin{document}

 \title{\Large{\textbf{$p$-Hyperbolic Zolotarev Functions in \\Boundary Value Problems for a $p\,$th order Differential Operator}}}

\author{\textbf{M.F. Bessmertny\u{\i}}\thanks{Department of Physics, V. N. Karazin Kharkov National University, 4 Svobody Sq, Kharkov, 61077, Ukraine.}
\and \textbf{V.A. Zolotarev}\thanks{B. Verkin Institute for Low Temperature Physics and Engineering of the National Academy of Sciences of Ukraine, 47 Nauky Ave., Kharkiv, 61103, Ukraine; Department of Higher Mathematics and Informatics,V. N. Karazin Kharkov National University, 4 Svobody Sq, Kharkov, 61077, Ukraine.}}



\date{}

\maketitle

\begin{abstract}
 For the self-adjoint operator of the $p$th derivative, a system of fundamental solutions is constructed. This system is analogues to the classical system of sines and cosines. The properties of such functions are studied. Classes of self-adjoint boundary conditions are described. For the operator of the third derivative, the resolvent is calculated and an orthonormal basis of eigenfunctions is given.
\end{abstract}
\medskip



    \section{\large{Introduction}}\label{s:1}

  The classical harmonic Fourier analysis is largely formed and developed by the spectral theory of the second-order differential operator (\cite{uj01} -- \cite{uj04}).  Mutually enriching connections of these two directions contributed to the development of elaborate mathematical apparatus. Fourier analysis serves as an important research tool in these (and not only these) areas of analysis. For differential operators of higher orders, this technique has not been developed in its entirety, only for some particular cases \cite{uj05}.

  Standard trigonometric functions are eigenfunctions of the simplest second-order differential operator. As far as we know, analogues of trigonometric and hyperbolic functions for operators of higher orders have not been considered before.

  In this article, for the $p\,$th order operator, a system of $p$ independent functions is proposed, which coincides with the system of sine and cosine for the case $p=2$. In addition, classes of self-adjoint boundary conditions are described in terms of such functions. For the operator of the third derivative, the resolvent is calculated and the orthonormal basis of eigenfunctions is presented. Note that the use of these functions for $p=3$ makes it possible to solve both the inverse spectral problem and the inverse scattering problem for a non-local potential \cite{uj06,uj07}.

   Importance of studying differential operators of order $p>2$ is determined, for example, by the fact that the search for $L-A$ pairs for Camassa-Holm and Degasperis-Procesi non-linear equations  (\cite{uj08}--\cite{uj12}) leads to a third-order operator $L$ (cubic string).

  In construction of $p$-hyperbolic ($p$-trigonometric) functions, the group of $p\,$th roots of unity plays an essential role. This group partitions the complex plane into $p$ sectors. The value of any $p$-function at the points of any sector is completely determined by its values at the points of one of the sectors.

  The article is organized as follows.

  In Section \ref{s:1}, the classical concept of even (odd) functions is generalized to the case $p>2$. In Proposition \ref{pro2.2}, the uniqueness of the decomposition of a complex-valued function into $p$ of $k_{p}$-even components is proved.

  $p$-hyperbolic functions are introduced in Section \ref{s:3} as a fundamental system of solutions to the simplest differential equation of $p\,$th order. The main relations between $p$-hyperbolic functions are given in Section \ref{s:4}. As an example, formulas for the case $p=3$ are given.

  For the case $p=3$, it is proved in section \ref{s:5} (Proposition \ref{pro5.1}) that the zeros of $3$-hyperbolic functions are located on the bisectors of each of the three sectors.

  Section \ref{s:6} gives a general solution to the inhomogeneous Cauchy problem.

  Section \ref{s:7} describes the classes of self-adjoint boundary conditions for the $p\,$th order differential operation. Note that in the even case $(p=2k)$ there exist separated boundary conditions, and in the odd case $(p=2k+1)$ there are no such boundary conditions.

  As an example of the use of $p$-functions,  an orthonormal basis is constructed in Section \ref{s:8} from the eigenvectors of a $3$rd order self-adjoint operator. In Section \ref{s:9}, its resolvent is calculated.

    \section{\large{$p\,$th roots of unity. $k_{p}$-even functions}}\label{s:2}

  Let $p\in \mathbb{N}$ and
      \begin{equation}\label{eq2.1}
      \zeta_{0}=1,\;\zeta_{1}=e^{i\frac{2\pi}{p}},\,\ldots\,,\zeta_{p-1}=e^{i\frac{2\pi(p-1)}{p}}
      \end{equation}
  be all the $p\,$th roots of unity.

  The multiplication of every complex number $z\in\mathbb{C}$ by
  the number $\zeta_{m}$ is the rotation of the complex plane around the origin through the angle $2\pi m/p$. The set of such rotations forms a group denoted by $G_{p}$.
  The numbers (\ref{eq2.1}) form a commutative group with multiplication as a group operation.
  We will identify this group with the group $G_{p}$.

  $G_{p}$ is a cyclic group. Each element of the group $G_{p}$ is a power of the element $\zeta_{1}$. If the number $p$ is prime, the set of powers $\zeta_{m}^{k}$, $k=0,1,\ldots,p-1$ of any element $\zeta_{m}\neq \zeta_{0}=1$ of the group $G_{p}$ is a permutation of the elements of the group.
  \medskip

  \noindent
  Note that $p\,$th roots of unity  satisfy the relations
    \begin{equation}\label{eq2.2}
    \zeta_{0}^{n}+\zeta_{1}^{n}+\cdots+\zeta_{p-1}^{n}=0, \quad\text{for every}\;n\in\mathbb{Z},\; n\neq 0\, (\text{mod}\,p),
    \end{equation}
    \begin{equation}\label{eq2.3}
    \zeta_{j}\zeta_{m}=\zeta_{j+m}\quad\text{where}\quad j+m:=(j+m)(\text{mod}\,p),
    \end{equation}
    \begin{equation}\label{eq2.4}
    \overline{\zeta}_{j}=\frac{1}{\zeta_{j}}=\zeta_{p-j}.
    \end{equation}
  Moreover, from (\ref{eq2.2}), (\ref{eq2.3}) we get
    \begin{equation}\label{eq2.5}
    \zeta_{m}^{0}+\zeta_{m}^{1}+\zeta_{m}^{2}+\cdots+\zeta_{m}^{p-1}=
      \begin{cases}
      p, & \text{if $m=0\,(\text{mod}\,p)$},\\
      0, & \text{if $m\neq0\,(\text{mod}\,p)$}.
      \end{cases}
    \end{equation}
  \medskip

  Let $X$ be some set and $Y$ be a linear space. Suppose $G$ is the group of transformations of the set $X$, and
     \begin{equation}\label{eq2.6}
     g\mapsto T_{g},\quad T(g_{1}g_{2})=T(g_{1})T(g_{2}),\quad \text{for any}\quad g_{1},g_{2}\in G
     \end{equation}
  is its representation by linear operators $T_{g}:Y\rightarrow Y$.
  \medskip

  The function $f:X\rightarrow Y$ is called \emph{invariant with respect to the representation} (\ref{eq2.6}) if
         $$
         T_{g}^{-1}f(g(x))=f(x)\quad\text{or}\quad f(g(x))=T_{g}f(x)\quad \text{for every}\quad g\in G.
         $$

    \begin{example}\label{exa1.1}
    Suppose $X=Y=\mathbb{R}$, and the two-element group $G=\{g_{1},\,g_{2}\}=\{1,\,-1\}$ is generated by the mirror reflection of the real axis $\mathbb{R}$ about the origin.
    Group $G$ has two linear representations in linear space $\mathbb{R}^{1}$:

     (i)  \emph{trivial} $g_{1}\mapsto 1,\, g_{2}\mapsto 1$,

     (ii) the \emph{identity} representation $g_{1}\mapsto 1,\,g_{2}\mapsto -1$.

    \noindent Functions on $\mathbb{R}$ that are invariant with respect to the representation (i) are called \emph{even},
    and those invariant with respect to the representation (ii) are called \emph{odd}. So the same group can generate several concepts of an invariant function.
    \end{example}

  \begin{proposition}\label{pro2.1}
  Let $p\in\mathbb{N}$ and $\zeta_{m}=e^{2\pi im/p}$, $m=0,1,\ldots,p-1$, be $p\,$th roots of unity. If $k\in\mathbb{N}$, then the map $\zeta_{m}\mapsto T_{p}^{k}(\zeta_{m}):\mathbb{C}\rightarrow \mathbb{C}$, where
     \begin{equation}\label{eq2.7}
     T_{p}^{k}(\zeta_{m})z=\zeta_{m}^{k}z,\quad z\in\mathbb{C},
     \end{equation}
  is a representation of the group $G_{p}$ in the linear space $\mathbb{C}^{1}$. For $k=0$, representation \emph{(\ref{eq2.7})} is trivial: $T_{p}^{0}(\zeta_{m})z\equiv z$.
  \end{proposition}
  \begin{proof}
  It's clear that $T_{p}^{k}(\zeta_{m}\zeta_{n})=T_{p}^{k}(\zeta_{m})T_{p}^{k}(\zeta_{n})$.
  \end{proof}

  A set $\Omega\subseteq\mathbb{C}$ is called \emph{$G_{p}$-invariant} if  $z\in\Omega$ implies $\zeta_{m}z\in\Omega$, for every $\zeta_{m}\in G_{p}$ and for every $z\in \Omega$. In particular, the complex plane $\mathbb{C}$ is a $G_{p}$-invariant set.

  \begin{definition}\label{def2.1}
     \emph{A complex-valued function $w=f(z)$ defined on a $G_{p}$-invariant set $\Omega$ is called $k_{p}$\emph{-even},\quad $k=0,1,\ldots,(p-1)$
     if it is invariant with respect to the representation $T_{p}^{k}$ (\ref{eq2.7}) of the group $G_{p}$, i. e.,}
        \begin{equation}\label{eq2.8}
           f(\zeta_{m}z)=\zeta_{m}^{k}f(z),\quad \text{\emph{for}}\quad m=0,1,\ldots,(p-1).
        \end{equation}
  \end{definition}

  \begin{remark}\label{rem2.1}
   For $p=2$, the introduced concept of a $k_{p}$-even function coincides with the standard one (see Example \ref{exa1.1}).
   \end{remark}

  We need a lemma.
  \begin{lemma}\label{lem2.1}
  Let $\zeta_{m}$ $(m=0,1,\ldots,p-1)$ be the $p\,$th roots of unity, where $p\in\mathbb{N}$. Then
    \begin{equation}\label{eq2.9}
    U=\frac{1}{\sqrt p}
      \begin{pmatrix}
          1           &     1       &        1        & \cdots &        1      \\
          1           &  \zeta_{1}  & \zeta_{1}^{2}   & \cdots & \zeta_{1}^{p-1}\\
          1           &  \zeta_{2}  & \zeta_{2}^{2}   & \cdots & \zeta_{2}^{p-1}\\
         \vdots       &   \vdots    &     \vdots      & \ddots &     \vdots     \\
          1           & \zeta_{p-1} & \zeta_{p-1}^{2} & \cdots & \zeta_{p-1}^{p-1}
      \end{pmatrix}
    \end{equation}
  is a unitary matrix, i. e., $UU^{\ast}=I$.
  \end{lemma}

  \begin{proof}
  We have
    $$
    1+\zeta_{j}\overline{\zeta}_{j}+\cdots+\zeta_{j}^{p-1}\overline{\zeta}_{j}^{p-1}=p,\quad\text{for}
    \quad j=0,1,\ldots,p-1,
    $$
  Since  $\overline{\zeta}_{j}=\zeta_{p-j}$, where $p-j:=(p-j)(\text{mod}\,p)$, then from (\ref{eq2.3}), (\ref{eq2.4}) and (\ref{eq2.5}) we obtain
    \begin{multline}\nonumber
    1+\zeta_{k}\overline{\zeta}_{j}+\cdots+\zeta_{k}^{p-1}\overline{\zeta}_{j}^{p-1}=
    1+\zeta_{k}\zeta_{p-j}+\cdots+\zeta_{k}^{p-1}\zeta_{p-j}^{p-1}=\\
     1+\zeta_{k+p-j}+\cdots+\zeta_{k+p-j}^{p-1}=0\quad\text{for}\quad k\neq j.\quad
    \end{multline}
  This implies $UU^{\ast}=I$.
  \end{proof}

  \begin{proposition}\label{pro2.2}
  Let $f(z)$ be a complex-valued function defined on a $G_{p}$-invariant set $\Omega\subseteq\mathbb{C}$. Then $f(z)$ is uniquely representable as
    \begin{equation}\label{eq2.10}
    f(z)=f_{0}(z)+f_{1}(z)+\cdots+f_{p-1}(z),
    \end{equation}
   where
    \begin{equation}\label{eq2.11}
    f_{k}(z)=\frac{1}{p}\sum_{m=0}^{p-1}\frac{1}{\zeta_{m}^{k}}f(\zeta_{m}z)
    \end{equation}
   is a $k_{p}$-even function, $k=0,1,\ldots,p-1$.
  \end{proposition}

  \begin{proof}
  We have
    $$
    f_{k}(\zeta_{j}z)=\frac{1}{p}\sum_{m=0}^{p-1}\frac{1}{\zeta_{m}^{k}}f(\zeta_{m}\zeta_{j}z)=
     \zeta_{j}^{k}\frac{1}{p}\sum_{m=0}^{p-1}\frac{1}{\zeta_{m+j}^{k}}f(\zeta_{m+j}z)=
     \zeta_{j}^{k}f_{k}(z).
    $$
  Moreover,
    \begin{equation}\nonumber
    \sum_{k=0}^{p-1}f_{k}(z)=\frac{1}{p}\sum_{k=0}^{p-1}
    \sum_{m=0}^{p-1}\frac{1}{\zeta_{m}^{k}}f(\zeta_{m}z)=
     \frac{1}{p}
    \sum_{m=0}^{p-1}\left(\sum_{k=0}^{p-1}\frac{1}{\zeta_{m}^{k}}\right)f(\zeta_{m}z).
    \end{equation}
  From (\ref{eq2.4}), (\ref{eq2.5}) we get
    $$
    \sum_{k=0}^{p-1}\frac{1}{\zeta_{m}^{k}}=\sum_{k=0}^{p-1}\overline{\zeta}_{m}^{k}=
      \begin{cases}
      p,& \text{if $m=0$,}\\
      0, & \text{if $m\neq0$.}
      \end{cases}
    $$
  Therefore,
    $$
    \sum_{k=0}^{p-1}f_{k}(z)=f(z).
    $$

  Next, prove the uniqueness. Suppose
    $$
    \sum_{k=0}^{p-1}f_{k}(z)=\sum_{k=0}^{p-1}g_{k}(z),
    $$
  where $f_{k}(z)$ and $g_{k}(z)$ are $k_{p}$-even functions. Then $k_{p}$-even functions $h_{k}(z)=f_{k}(z)-g_{k}(z)$ $\left(h_{k}(\zeta_{j}z\right)=\zeta_{j}^{k}h_{k}(z))$ satisfy the system of equations
    $$
    \begin{cases}
    h_{0}(z)+h_{1}(z)+h_{2}(z)+\cdots+h_{p-1}(z)=0\qquad\qquad\qquad\\
    h_{0}(z)+\zeta_{1}h_{1}(z)+\zeta_{1}^{2}h_{2}(z)+\cdots+\zeta_{1}^{p-1}h_{p-1}(z)=0\quad\quad\\
    \cdots\quad\cdots\quad\cdots\\
    h_{0}(z)+\zeta_{p-1}h_{1}(z)+\zeta_{p-1}^{2}h_{2}(z)+\cdots+\zeta_{p-1}^{p-1}h_{p-1}(z)=0
    \end{cases}
    $$
  By Lemma \ref{lem2.1}, its matrix of coefficients is nonsingular. Then
  $h_{k}(z)\equiv 0$, $k=0,1,\ldots,p-1$.
  \end{proof}

   \begin{corollary}\label{cor2.1}
   Let $p\in\mathbb{N}$. If a $k_{p}$-even function $f_{k}(z)$ is analytic in a neighborhood of zero, then its Taylor series has the form
               \begin{equation}\label{eq2.12}
               f_{k}(z)=\sum_{n=0}^{\infty}a_{n}z^{k+np}.
               \end{equation}
   \end{corollary}
   \begin{proof}
   Due to the uniqueness of the decomposition (\ref{eq2.10}), it is sufficient to apply the formula (\ref{eq2.11}) to the Taylor series of the function $f_{k}(z)$.
   \end{proof}

   \begin{proposition}\label{pro2.3}
   Let $f(z)$ be a differentiable $k_{p}$-even function. Then its derivative
      $$
      \frac{df}{dz}(z)
      $$
   is a $(k-1)_{p}$-even function.
   \end{proposition}

   \begin{proof}
   It is given that $f(z)=(1/\zeta_{m}^{k})f(\zeta_{m}z)$. Then
     $$
     \frac{df}{dz}(z)=\frac{1}{\zeta_{m}^{k}}\frac{df}{dz}(\zeta_{m}z)\zeta_{m}=
       \frac{1}{\zeta_{m}^{k-1}}\frac{df}{dz}(\zeta_{m}z).
     $$
     \end{proof}

   \section{\large{Fundamental systems of solutions}}\label{s:3}

   We are interested in solutions of the differential equation
     \begin{equation}\label{eq3.1}
     \left(\frac{1}{i}D\right)^{p}y(x)=\lambda^{p}y(x)\quad (x\in\mathbb{R},\;D=d/dx,\;\lambda\in\mathbb{C}),
     \end{equation}
   where $p$ is a natural number.

   \begin{proposition}\label{pro3.1}
   Any solution of the equation \emph{(\ref{eq3.1})} has the form
     $$
     y(x)=s(i\lambda x),
     $$
   where $s(z)$ is the solution of the equation
     \begin{equation}\label{eq3.2}
     D^{p}s(z)=s(z),\quad (z\in\mathbb{C},\;D=d/dz).
     \end{equation}
   \end{proposition}

   \begin{proof}
   Introducing a new function $s(z)=y(x)$, where $z=i\lambda x$, we obtain
     \begin{multline}\nonumber
     \lambda^{p}s(z)=\lambda^{p}y(x)=(-i)^{p}\frac{d^{p}}{dx^{p}}y(x)=(-i)^{p}\frac{d^{p}}{dx^{p}}s(z)=\\
     (-i)^{p}\frac{d^{p}}{dz^{p}}s(z)\left(\frac{dz}{dx}\right)^{p}=
       (-i)^{p}i^{p}\lambda^{p}\frac{d^{p}}{dz^{p}}s(z)=\lambda^{p}\frac{d^{p}}{dz^{p}}s(z).
     \end{multline}
   This implies (\ref{eq3.2}). It is obvious that (\ref{eq3.2}) implies (\ref{eq3.1}).
   \end{proof}

   The function $s(z)=e^{z}$ is the only solution (up to a multiplicative constant) of equations (\ref{eq3.2}) for all $p\in\mathbb{N}$. For each fixed $p\in\mathbb{N}$, the solution $s(z)=e^{z}$ is called \emph{a principal solution} of the equation (\ref{eq3.2}).

   \begin{definition}\label{def3.1}
   \emph{Let $p$ be a natural number, and $\{\zeta_{k}\}_{k=0}^{p-1}$ be $p\,$th roots of unity. $k_{p}$-even components
      \begin{equation}\label{eq3.3}
      \boxed{s_{k}(z)=\frac{1}{p}\sum_{m=0}^{p-1}\frac{1}{\zeta_{m}^{k}}e^{z\zeta_{m}},\quad k=0,1,\ldots,p-1}
      \end{equation}
   of the function $e^{z}$ are called} $p$-hyperbolic functions.
   \end{definition}

     \begin{remark}\label{rem3.1}
   The name ``$p$-hyperbolic functions" is due to the fact that, for  $p=2$ ($\zeta_{0}=1$, $\zeta_{1}=-1$), we have
       $$
       s_{0}(z)=\frac{1}{2}(e^{z}+e^{-z})=\cosh z,\quad
       s_{1}(z)=\frac{1}{2}(e^{z}-e^{-z})=\sinh z.
       $$
   By analogy to this case, we call
     \begin{equation}\label{eq3.4}
     c_{k}(z)=\frac{1}{i^{k}}s_{k}(iz),\quad k=0,1,\ldots,p-1
     \end{equation}
   \emph{a $p$-trigonometric functions}. By virtue of (\ref{eq3.4}), it suffices to study only the properties of $p$-hyperbolic functions.
     \end{remark}

   Note that according to Proposition \ref{pro2.2},
    $$
    s_{0}(z)+s_{1}(z)+\cdots+s_{p-1}(z)=e^{z}.
    $$
  This identity is a generalization of Euler's formula for the case $p\neq 2$.

     \begin{proposition}\label{pro3.2}
   $p$-hyperbolic functions $s_{k}(z)$, $k=0,1,\ldots,p-1$ are entire functions of exponential type. These functions
   \smallskip

   \noindent
   \emph{(i)}  form a fundamental system of solutions of the  equation \emph{(\ref{eq3.2})},
   \smallskip

   \noindent
   \emph{(ii)} satisfy the initial data
     \begin{equation}\label{eq3.5}
     D^{n}s_{k}(z)|_{z=0}=
       \begin{cases}
       1, & \text{if $n=k$,}\\
       0, & \text{if $n\neq k$,}
       \end{cases}\quad k,n=0,1,\ldots,p-1,
     \end{equation}

   \noindent
   \emph{(iii)} have derivatives
     \begin{equation}\label{eq3.6}
     Ds_{0}(z)=s_{p-1}(z),\qquad Ds_{k}(z)=s_{k-1}(z),\quad k=1,\ldots,p-1,
     \end{equation}

   \noindent
   \emph{(iv)} are represented by the Taylor series
     \begin{equation}\label{eq3.7}
     s_{k}(z)=\sum_{n=0}^{\infty}\frac{z^{k+pn}}{(k+pn)!},\quad k=0,1,\ldots,p-1,
     \end{equation}

   \noindent
   \emph{(v)} are real for real $z$, i. e.,
      \begin{equation}\label{eq3.7}
        \overline{s_{k}(z)}=s_{k}(\overline{z}),\quad k=0,1,\ldots,p-1,
      \end{equation}

   \noindent
   \emph{(vi)}
     \begin{equation}\label{eq3.9}
     \det W[s_{k}(z)]=
       \left|\begin{array}{cccc}
           s_{0}(z)     &       s_{1}(z)   & \cdots & s_{p-1}(z)   \\
           s'_{0}(z)    &       s'_{1}(z)  & \cdots & s'_{p-1}(z)  \\
             \vdots     &       \vdots     & \ddots & \vdots       \\
       s_{0}^{(p-1)}(z) & s_{1}^{(p-1)}(z) & \cdots & s^{(p-1)}_{p-1}(z)
       \end{array}\right|\equiv 1.
     \end{equation}
   \end{proposition}

   \begin{proof}
   (ii), (iii) and (iv) follow from Definition \ref{def3.1} and relations (\ref{eq2.2}) -- (\ref{eq2.5}). (v) follows from (iv).

   \noindent
   (vi). Given (iii), we get
       \begin{equation}\label{eq3.10}
        W[s_{k}(z)]=
           \begin{pmatrix}
           s_{0}(z)    &       s_{1}(z)    & \cdots & s_{p-1}(z)   \\
           s_{p-1}(z)  &       s_{0}(z)    & \cdots & s_{p-2}(z)  \\
             \vdots    &       \vdots      & \ddots & \vdots       \\
           s_{1}(z)    &       s_{2}(z)    & \cdots & s_{0}(z)
           \end{pmatrix}=U\Lambda(z)U^{\ast},\qquad
       \end{equation}
   where the unitary matrix $U$ is defined by the relation (\ref{eq2.9}) and
       $$
       \Lambda(z)=\text{diag}\{e^{z\zeta_{0}},e^{z\zeta_{1}},\ldots,e^{z\zeta_{p-1}}\}.
       $$
   Since
      $$
      \det \Lambda(z)=e^{z(\zeta_{0}+\zeta_{1}+\cdots+\zeta_{p-1})}=e^{0}=1,
      $$
   then $\det W[s_{k}(z)]\equiv 1$.
   \smallskip

   \noindent
   (i). From (iii) yields $D^{p}s_{k}(z)=s_{k}(z)$, $k=0,1,\ldots,p-1$, i. e., $s_{k}(z)$ are solutions of equation (\ref{eq3.2}). (vi) implies linear independence of functions. Then $\{s_{k}(z)\}_{k=0}^{p-1}$ is the fundamental system of solutions of the equation (\ref{eq3.2}).
   \end{proof}

   \begin{corollary}\label{cor3.1}
   The function
     \begin{equation}\label{eq3.11}
     y(z)=y_{0}s_{0}(z)+y_{1}s_{1}(z)+\cdots+y_{p-1}s_{p-1}(z)
     \end{equation}
   is solution of Cauchy problem
     \begin{equation}\label{eq3.12}
     \begin{array}{c}
     D^{p}y(z)=y(z),\quad (D=d/dz)\\
     \\y(0)=y_{0},\;y'(0)=y_{1},\;\ldots\;,y^{(p-1)}(0)=y_{p-1}.
     \end{array}
     \end{equation}
   \end{corollary}

   \begin{corollary}\label{cor3.2}
   For the solution \emph{(\ref{eq3.11})} of the Cauchy problem \emph{(\ref{eq3.12})} we have
     \begin{equation}\nonumber
     \det W[y,y',\ldots,y^{(p-1)}]=
     \left|\begin{array}{cccc}
              y(z)      &  y'(z) &  \cdots & y^{(p-1)}(z) \\
          y^{(p-1)}(z)  &  y(z)  &  \cdots & y^{(p-2)}(z) \\
            \vdots      & \vdots &  \ddots &    \vdots    \\
              y'(z)     & y''(z) &  \cdots &     y(z)
           \end{array}
           \right|=C,
     \end{equation}
   where the constant $C$ depends only on the initial data and does not depend on $z$.
   \end{corollary}

   \begin{proof}
   Since $W[y,y',\ldots,y^{(p-1)}]=Y_{p}^{T}\cdot W[s_{k}(z)]$, where
      $$
      Y_{p}=\begin{pmatrix}
                y_{0}   &  y_{1} &  \cdots &  y_{p-1} \\
               y_{p-1}  &  y_{0} &  \cdots & y_{(p-2)} \\
                \vdots  & \vdots &  \ddots &  \vdots    \\
                y_{1}   &  y_{2} &  \cdots &   y_{0}
            \end{pmatrix},
      $$
   and $\det W[s_{k}(z)]\equiv 1$, then $\det W(y,y',\ldots,y^{(p-1)})=\det Y_{p}$.
   \end{proof}

   \section{\large{Basic identities for $p$-hyperbolic functions}}\label{s:4}

   The derivation of relations between $p$-hyperbolic functions is based on the properties of $p\,$roots of unity and the following simple statement.
     \begin{proposition}\label{pro4.1}
     Let $\zeta_{m}$ $(m=0,1,\ldots,p-1)$ be $p\,$roots of unity.
      The maps
      \begin{equation}\label{eq4.1}
      z\mapsto W_{p}(z)=
            \begin{pmatrix}
               s_{0}(z)   &  s_{1}(z)  & \cdots & s_{p-1}(z)  \\
              s_{p-1}(z)  &  s_{0}(z)  & \cdots & s_{p-2}(z)  \\
                \vdots    &   \vdots   & \ddots &   \vdots    \\
               s_{1}(z)   &  s_{2}(z)  & \cdots & s_{0}(z)
            \end{pmatrix}
      \end{equation}
      is a representation of the additive group of complex numbers in the linear space $\mathbb{C}^{p}$. Moreover, by Proposition \emph{\ref{pro3.2} (vi)}, the main identity
        \begin{equation}\label{eq4.2}
        \det W_{p}(z)\equiv 1
        \end{equation}
      holds.
      \end{proposition}

     \begin{proof}
     Recall (see (\ref{eq3.10})) that
       \begin{equation}\label{eq4.3}
       W_{p}(z)=U\Lambda(z)U^{\ast},\quad
         \Lambda(z)=\text{diag}\{e^{z\zeta_{0}},e^{z\zeta_{1}},\ldots,e^{z\zeta_{p-1}}\},
       \end{equation}
     where $U$ is a unitary matrix from Lemma \ref{lem2.1}. Since
         $$
         \Lambda(z_{1}+z_{2})=\Lambda(z_{1})\cdot\Lambda(z_{2}),\quad \Lambda(0)=I,
         $$
     then $W_{p}(z_{1}+z_{2})=W_{p}(z_{1})\cdot W_{p}(z_{2})$,\; $W_{p}(0)=I$.
     \end{proof}

     \begin{remark}\label{rem3.1}
     Since the elements of the matrix $W_{p}(z_{1}+z_{2})$ are functions $s_{k}(z_{1}+z_{2})$, then from the identity $W_{p}(z_{1}+z_{2})=W_{p}(z_{1})\cdot W_{p}(z_{2})$ we obtain formulas for functions of the sum of the arguments. The identity $W_{p}(z)W_{p}(-z)=I$ binds the values of functions on opposite values of the arguments.
     \end{remark}

     \begin{remark}\label{rem3.2}
     From (\ref{eq4.3}) follows $\Lambda(z)=U^{\ast}W_{p}(z)U$. Comparing the diagonal elements, we obtain \emph{the Euler identities} for $p$-hyperbolic functions
        \begin{equation}\label{eq4.4}
        e^{z\zeta_{k}}=  \zeta_{0}^{k}s_{0}(z)+\zeta_{1}^{k}s_{1}(z)+\cdots+\zeta_{p-1}^{k}s_{p-1}(z),\quad k=0,1,\ldots,p-1.
        \end{equation}
     \end{remark}

   To conclude this section, we present the corresponding formulas for the case $p=3$. We obtain
   \medskip

     (i) \emph{the main identity} $(\det W_{3}(z)=1)$
       \begin{equation}\label{eq4.5}
       s_{0}(z)^{3}+s_{1}(z)^{3}+s_{2}(z)^{3}-3s_{0}(z)s_{1}(z)s_{2}(z)=1,
       \end{equation}

     (ii) \emph{addition formulas} $(W_{3}(z+w)=W_{3}(z)\cdot W_{3}(w))$
       \begin{equation}\label{eq4.6}
         \begin{array}{ccc}
         s_{0}(z+w) & = & s_{0}(z)s_{0}(w)+s_{1}(z)s_{2}(w)+s_{2}(z)s_{1}(w),\\
         s_{1}(z+w) & = & s_{0}(z)s_{1}(w)+s_{1}(z)s_{0}(w)+s_{2}(z)s_{2}(w),\\
         s_{2}(z+w) & = & s_{0}(z)s_{2}(w)+s_{1}(z)s_{1}(w)+s_{2}(z)s_{0}(w),
         \end{array}
       \end{equation}

     (iii) \emph{double argument}
       \begin{equation}\label{eq4.7}
         \begin{array}{ccc}
         s_{0}(2z) & = & s_{0}(z)^{2}+2s_{1}(z)s_{2}(z),\\
         s_{1}(2z) & = & 2s_{0}(z)s_{1}(z)+s_{2}(z)^{2},\\
         s_{2}(2z) & = & 2s_{0}(z)s_{2}(z)+s_{1}(z)^{2},
         \end{array}
       \end{equation}

     (iv) \emph{triple argument}
       \begin{equation}\label{eq4.8}
         \begin{array}{ccc}
         s_{0}(3z) & = & 1+9s_{0}(z)s_{1}(z)s_{2}(z),\\
         s_{1}(3z) & = & 3s_{0}(z)^{2}s_{1}(z)+3s_{0}(z)^{2}s_{2}(z)+3s_{1}(z)^{2}s_{2}(z),\\
         s_{2}(3z) & = & 3s_{0}(z)s_{1}(z)^{2}+3s_{0}(z)s_{2}(z)^{2}+3s_{1}(z)s_{2}(z)^{2},
         \end{array}
       \end{equation}

     (v) \emph{convert product to sum}
       \begin{equation}\label{eq4.9}
     \begin{array}{ccc}
     3s_{0}(z)s_{0}(w) &=& s_{0}(z+w)+s_{0}(z+\zeta_{1}w)+s_{0}(z+\zeta_{2}w), \\
     3s_{1}(z)s_{2}(w) &=& s_{0}(z+w)+\zeta_{1}s_{0}(z+\zeta_{1}w)+\zeta_{2}s_{0}(z+\zeta_{2}w), \\
     3s_{1}(z)s_{0}(w) &=& s_{1}(z+w)+s_{1}(z+\zeta_{1}w)+s_{1}(z+\zeta_{2}w), \\
     3s_{2}(z)s_{2}(w) &=& s_{1}(z+w)+\zeta_{1}s_{1}(z+\zeta_{1}w)+\zeta_{2}s_{1}(z+\zeta_{2}w), \\
     3s_{2}(z)s_{0}(w) &=& s_{2}(z+w)+s_{2}(z+\zeta_{1}w)+s_{2}(z+\zeta_{2}w), \\
     3s_{1}(z)s_{1}(w) &=& s_{2}(z+w)+\zeta_{2}s_{2}(z+\zeta_{1}w)+\zeta_{1}s_{1}(z+\zeta_{2}w).
     \end{array}
       \end{equation}

     (vi) \emph{connection of direct and opposite arguments} $(W_{3}(z)W_{3}(-z)=I)$
       \begin{equation}\label{eq4.10}
         \begin{array}{ccc}
          s_{0}(z)s_{0}(-z)+s_{1}(z)s_{2}(-z)+s_{2}(z)s_{1}(-z) & = & 1,\\ s_{0}(z)s_{1}(-z)+s_{1}(z)s_{0}(-z)+s_{2}(z)s_{2}(-z) & = & 0,\\
          s_{0}(z)s_{2}(-z)+s_{1}(z)s_{1}(-z)+s_{2}(z)s_{0}(-z) & = & 0,
         \end{array}
       \end{equation}

     (vii) \emph{opposite argument}
       \begin{equation}\label{eq4.11}
         \begin{array}{ccc}
          s_{0}(-z) & = & s_{0}(z)^{2}-s_{1}(z)s_{2}(z)\\
          s_{1}(-z) & = & s_{2}(z)^{2}-s_{0}(z)s_{1}(z)\\
          s_{2}(-z) & = & s_{1}(z)^{2}-s_{0}(z)s_{2}(z).
         \end{array}
       \end{equation}
     From here it is not difficult to obtain formulas for $p$-trigonometric functions $c_{k}(z)=s_{k}(iz)/i^{k}$.

   \section{\large{Zeros of functions $s_{k}(z)$ for $p=3$}}\label{s:5}
   Let $\zeta_{\nu}$, $\nu=0,1,2$ be $3$-roots of unity. Introduce the notation
     $$
     l_{\nu}^{+}=\{z\in\mathbb{C}\mid z=t\cdot\zeta_{\nu},\;t>0\},\quad
     l_{\nu}^{-}=\{z\in\mathbb{C}\mid z=t\cdot\zeta_{\nu},\;t\leq0\}.
     $$
   Rays $l_{\nu}^{+}$, $\nu=0,1,2$ divide the complex plane $\mathbb{C}$ into three sectors
     $$
     S_{_{\nu}}=\{z\in\mathbb{C}\mid \arg\zeta_{\nu}<\arg z<\arg\zeta_{(\nu+1)\,(\text{mod}\,3)}\},\quad \nu=0,1,2.
     $$
   The ray $l_{\nu}^{-}$ is a bisectrix of sector $S_{(\nu+1)\,(\text{mod}\,3)}$.

   \begin{proposition}\label{pro5.1}
   $3$-hyperbolic functions $s_{k}(z)$ $(k=0,1,2)$ do not vanish on the set
     \begin{equation}\label{eq5.1}
     \Omega=\mathbb{C}\setminus\bigcup_{\nu=0}^{2}l_{\nu}^{-}.
     \end{equation}
   In other words, the zeros of the functions $s_{k}(z)$ lie only on the bisectrices of each sector.
   \end{proposition}
  We need two lemmas

  \begin{lemma}\label{lem5.1}
  Let $z=x+iy$ $(x,y\in\mathbb{R})$ and $s_{k}(z)$, $k=0,1,2$ be $3$-hyperbolic functions. The zeros of the function $s_{k}(x+iy)$ are the roots of the system of equations
    \begin{equation}\label{eq5.2}
     \begin{cases}
     \cosh\frac{y\sqrt 3}{2}\cos\left(\frac{x\sqrt 3}{2}-\frac{2\pi k}{3}\right)=
       -\frac{1}{2}e^{\frac{3x}{2}}\cos\frac{3y}{2}\\
     \sinh\frac{y\sqrt 3}{2}\sin\left(\frac{x\sqrt 3}{2}-\frac{2\pi k}{3}\right)=
       \frac{1}{2}e^{\frac{3x}{2}}\sin\frac{3y}{2}
     \end{cases}
     \end{equation}
  \end{lemma}

  \begin{proof}
  We have
      $$
      \zeta_{k}=\zeta_{1}^{k}=e^{i2\pi k/3}\;\Rightarrow\; \zeta_{0}=1,\quad\zeta_{1}=-\frac{1}{2}+i\frac{\sqrt 3}{2},\quad
      \zeta_{2}=-\frac{1}{2}-i\frac{\sqrt 3}{2}.
      $$
  Then from (\ref{eq3.3}) we obtain
      \begin{multline}\nonumber
      3s_{k}(z)=e^{z}+\frac{1}{\zeta_{1}^{k}}e^{z\zeta_{1}}+\frac{1}{\zeta_{2}^{k}}e^{z\zeta_{2}}=
        e^{z}+\overline{\zeta}_{1}^{k}e^{z\zeta_{1}}+\zeta_{1}^{k}e^{z\zeta_{2}}=\\
      e^{-\frac{z}{2}}\left(e^{\frac{3z}{2}}+e^{-i\frac{2\pi k}{3}}e^{i\frac{z\sqrt 3}{2}}+
      e^{i\frac{2\pi k}{3}}e^{-i\frac{z\sqrt 3}{2}}\right)=\qquad\qquad\\
      e^{-\frac{z}{2}}\left(e^{\frac{3z}{2}}+e^{i\left(\frac{z\sqrt 3}{2}-\frac{2\pi k}{3}\right)}+
      e^{-i\left(\frac{z\sqrt 3}{2}-\frac{2\pi k}{3}\right)}\right)=\\
        e^{-\frac{z}{2}}\left(e^{\frac{3z}{2}}+2\cos\left(\frac{z\sqrt 3}{2}-\frac{2\pi k}{3}\right)\right).
      \end{multline}
  Therefore, the zeros of the function $s_{k}(z)$ are the roots of the equation
     \begin{equation}\label{eq5.3}
     \cos\left(\frac{z\sqrt 3}{2}-\frac{2\pi k}{3}\right)+\frac{1}{2}e^{\frac{3z}{2}}=0.
     \end{equation}
  It is enough to apply the identity
     $$
     \cos(\alpha+i\beta)=\cosh\beta\cos\alpha-i\sinh\beta\sin\alpha
     $$
  to the left side (\ref{eq5.3}) and separate the real and imaginary parts.
  \end{proof}

  \begin{lemma}\label{lem5.2}
  For $y\geq0$ the inequality
     \begin{equation}\label{eq5.4}
     \sqrt 3\cdot\tanh\left(\frac{y\sqrt 3}{2}\right)\geq\sin\left(\frac{3y}{2}\right)
     \end{equation}
  is true. Equality in \emph{(\ref{eq5.4})} holds only for $y=0$.
  \end{lemma}

  \begin{proof}
  The hyperbolic tangent is represented by an alternating series
    $$
    \tanh x=x-\frac{1}{3}x^{3}+\frac{2}{15}x^{5}-\frac{17}{315}x^{7}+\cdots\quad\text{for}\quad
    |x|<\frac{\pi}{2}.
    $$
  Therefore,
    $$
    \tanh x=x-\frac{1}{3}x^{3}+\Delta,\quad\text{where}\quad \Delta>0.
    $$
  From this we obtain
    \begin{equation}\label{eq5.5}
    \sqrt 3\cdot\tanh\left(\frac{y\sqrt 3}{2}\right)\geq\frac{3y}{2}-\frac{3y^{3}}{8}\quad
    \text{for}\quad 0\leq y\leq\frac{\pi}{\sqrt 3}.
    \end{equation}
  Similarly,
   $$
   \sin x=x-\frac{1}{3!}x^{3}+\frac{1}{5!}x^{5}-\Delta_{1}\quad\text{where}\quad \Delta_{1}\geq 0
   \quad\text{for}\quad x\geq 0.
   $$
  Then
    \begin{equation}\label{eq5.6}
    \frac{3y}{2}-\frac{9y^{3}}{16}+\frac{81y^{5}}{1280}\geq\sin\left(\frac{3y}{2}\right)
    \quad\text{for}\quad y\geq 0.
    \end{equation}
  Given (\ref{eq5.5}) and (\ref{eq5.6}), the inequality (\ref{eq5.4}) is true if
    \begin{equation}\label{eq5.7}
    \frac{3y}{2}-\frac{3y^{3}}{8}>\frac{3y}{2}-\frac{9y^{3}}{16}+\frac{81y^{5}}{1280}
    \end{equation}
  is true. Find a solution to the inequality (\ref{eq5.7}).
   We have
     $$
     y^{3}(80-27y^{2})>0\quad\Rightarrow\quad 0<y<y_{0}=\frac{4\sqrt 5}{3\sqrt 3}\approx 1,72<\frac{\pi}{\sqrt 3}\approx 1,81.
     $$
  Therefore, the strict inequality (\ref{eq5.4}) is true for $0<y<y_{0}$. Since
     $$
     \sqrt 3\tanh\left(\frac{y\sqrt 3}{2}\right)\geq
       \sqrt 3\tanh\left(\frac{y_{0}\sqrt 3}{2}\right)=
       \sqrt 3\tanh\left(\frac{2\sqrt 5}{3}\right)\approx 1,56>1
     $$
  for $y\geq y_{0}$, the strict inequality (\ref{eq5.4}) is true for all $y>0$.
  \end{proof}
  \medskip

  \noindent
  \emph{Proof of Proposition} \ref{pro5.1}.\quad
  Since $s_{k}(\zeta_{\nu}z)=\zeta_{\nu}^{k}s_{k}(z)$, it suffices to prove that in some closed sector $\overline{S}_{\nu}$ (minus the bisectrix) the function $s_{k}(z)$ does not vanish.

  If the point $z=x+iy$ belongs to $\overline{S}_{1}\setminus l_{0}^{-}$, then $x\leq 0$ and $y\neq 0$, $y\in\mathbb{R}$. Taking Lemma \ref{lem5.1} into account, we prove that such points do not satisfy the system of equations (\ref{eq5.2}). From the first equation of the system, we obtain
     $$
     \cos\left(\frac{x\sqrt 3}{2}-\frac{2\pi k}{3}\right)=
         -\frac{e^{3x/2}}{2}\frac{\cos(3y/2)}{\cosh(y\sqrt 3/2)}.
     $$
  Thus
    \begin{equation}\label{eq5.8}
    \sin^{2}\left(\frac{x\sqrt 3}{2}-\frac{2\pi k}{3}\right)=
      \frac{4\cosh^{2}(y\sqrt 3/2)-e^{3x}\cos^{2}(3y/2)}{4\cosh^{2}(y\sqrt 3/2)}.
    \end{equation}
  Show that (\ref{eq5.8}) does not satisfy the second equation from (\ref{eq5.2}) for
  $x\leq 0$ and $y\neq 0$. Squaring both sides of the second equation, we get
    $$
    \sinh^{2}\frac{y\sqrt 3}{2}\sin^{2}\left(\frac{x\sqrt 3}{2}-\frac{2\pi k}{3}\right)=
        \frac{1}{4}e^{3x}\sin^{2}\frac{3y}{2}.
    $$
  Substituting (\ref{eq5.8}) into the last equation, we obtain
    \begin{equation}\label{eq5.9}
    \tanh^{2}\frac{y\sqrt 3}{2}
    \left(4\cosh^{2}\frac{y\sqrt 3}{2}-e^{3x}\cos^{2}\frac{3y}{2}\right)=e^{3x}\sin^{2}\frac{3y}{2}.
    \end{equation}
  By Lemma \ref{lem5.2} we have the strict inequality
     $$
     3\cdot\tanh^{2}\frac{y\sqrt 3}{2}>\sin^{2}\frac{3y}{2}\quad\text{for}\quad y\neq 0.
     $$
  Then from (\ref{eq5.9}) we obtain the  strict inequality (for $x\leq 0$)
    \begin{multline}\nonumber
    \tanh^{2}\frac{y\sqrt 3}{2}
    \left(4\cosh^{2}\frac{y\sqrt 3}{2}-e^{3x}\cos^{2}\frac{3y}{2}\right)\geq\\
         3\cdot\tanh^{2}\frac{y\sqrt 3}{2}>\sin^{2}\frac{3y}{2}\geq
            e^{3x}\sin^{2}\frac{3y}{2}.\qquad
    \end{multline}
  Therefore, $s_{k}(z)\neq 0$ for $z\in \overline{S}_{1}\setminus l_{0}^{-}$.
  \qed
  \medskip

  By Proposition \ref{pro5.1}, the zeros of the $s_{0}(z),s_{1}(z)$ and $s_{2}(z)$  are located on $l_{\nu}^{-}$, $\nu=0,1,2$.
  Since $s_{k}(\zeta_{m}z)=\zeta_{m}^{k}s_{k}(z)$ are satisfied, it suffices to find the zeros lying only on $l_{0}^{-}=\{x\in\mathbb{R}\mid x\leq 0\}$.

  Taking into account (\ref{eq5.3}), we obtain the following assertion.

    \begin{proposition}\label{pro5.2}
  An infinite sequence of zeros $x_{j}^{(k)}$ $(j=1,2,\,\ldots)$ of the function $s_{k}(x)$ are the roots of the equation
       \begin{equation}\label{eq5.10}
       \cos\left(\frac{x\sqrt 3}{2}-\frac{2\pi k}{3}\right)=-\frac{1}{2}e^{\frac{3x}{2}},\quad
       k=0,1,2.
       \end{equation}
    \end{proposition}
  The sequence $x_{j}^{(2)}$ interlaces with the sequence $x_{j}^{(0)}$ which, in its turn, interlaces with the sequence $x_{j}^{(1)}$. Note that $x_{1}^{(2)}=x_{1}^{(1)}=0$ and $x_{1}^{(0)}<0$.

  \begin{corollary}\label{cor5.1}
  If $x\rightarrow -\infty$, then asymptotically the zeros of the function $s_{k}(x)$ tend to the roots of the equation
      \begin{equation}\label{eq5.11}
      \cos\left(\frac{x\sqrt 3}{2}-\frac{2\pi k}{3}\right)=0.
      \end{equation}
  \end{corollary}

   \section{\large{General solution of the inhomogeneous equation}}\label{s:6}

   Consider the Cauchy problem
     \begin{equation}\label{eq6.1}
     \left(\frac{1}{i}D\right)^{p}y(\lambda,x)=\lambda^{p}y(\lambda,x)+f(x),\quad x\in(0,l)
     \end{equation}
   with initial data at zero
      \begin{equation}\label{eq6.2}
      y(\lambda,0)=y_{0},\quad y'(\lambda,0)=y_{1},\quad\ldots\quad,\; y^{(p-1)}(\lambda,0)=y_{p-1},
      \end{equation}
   where $f(x)$ is the function of $x\in\mathbb{R}$, $y_{s}\in\mathbb{C}$ $(0\leq s\leq p-1)$ and $0<l<\infty$.

   It is easy to see that the solution $y_{0}(\lambda,x)$ of the problem (\ref{eq6.1}), (\ref{eq6.2}) with $f(x)\equiv 0$ has the form
     \begin{equation}\label{eq6.3}
     y_{0}(\lambda,x)=\sum_{k=0}^{p-1}y_{k}\frac{1}{(i\lambda)^{k}}s_{k}(i\lambda x)=
         \sum_{k=0}^{p-1}C_{k}e^{i\lambda\zeta_{k} x}
     \end{equation}
   For $f(x)\neq 0$ we find the solution $y_{1}(\lambda,x)$ of equation (\ref{eq6.1}) with zero initial data $y_{0}=y_{1}=\cdots=y_{p-1}=0$. For the method of variation of arbitrary constants, we set
     \begin{equation}\label{eq6.4}
     y_{1}(\lambda,x)=\sum_{k=0}^{p-1}C_{k}(x)e^{i\lambda\zeta_{k} x},
     \end{equation}
   where $C_{k}(x)$ are smooth functions. Easy to see that
     $$
     \left(\frac{1}{i}D\right)y_{1}(\lambda,x)=
     \lambda\sum_{k=0}^{p-1}C_{k}(x)\zeta_{k}e^{i\lambda\zeta_{k} x},
     $$
   under the additional condition
     \begin{equation}\label{eq6.5}
     \sum_{k=0}^{p-1}C'_{k}(x)e^{i\lambda\zeta_{k} x}\equiv 0.
     \end{equation}
   After the second differentiation, under the additional condition
     \begin{equation}\label{eq6.6}
     \sum_{k=0}^{p-1}C'_{k}(x)\zeta_{k}e^{i\lambda\zeta_{k} x}\equiv 0,
     \end{equation}
   we get
     $$
     \left(\frac{1}{i}D\right)^{2}y_{1}(\lambda,x)=
     \lambda^{2}\sum_{k=0}^{p-1}C_{k}(x)\zeta_{k}^{2}e^{i\lambda\zeta_{k} x}.
     $$
   Repeating this trick, at the $p$\,th step we get
     $$
     \left(\frac{1}{i}D\right)^{p}y_{1}(\lambda,x)=\lambda^{p}y_{1}(\lambda,x)+
     \frac{1}{i}\lambda^{p-1}\sum_{k=0}^{p-1}C'_{k}(x)\zeta_{k}^{p-1}e^{i\lambda\zeta_{k} x}.
     $$

   The function $y_{1}(\lambda,x)$ (\ref{eq6.4}) will be a solution to equation (\ref{eq6.1}) if $C_{k}(x)$ satisfy additional conditions (of the form (\ref{eq6.5}), (\ref{eq6.6}) accepted at the first $p-1$ steps) and the condition
     \begin{equation}\label{eq6.7}
     \sum_{k=0}^{p-1}C'_{k}(x)\zeta_{k}^{p-1}e^{i\lambda\zeta_{k} x}=if(x)/\lambda^{p-1}.
     \end{equation}
   Taking into account (\ref{eq6.5}), (\ref{eq6.6}), and (\ref{eq6.7}), we obtain a system of equations for the $\{C'_{k}(x)\}_{k=0}^{p-1}$:
     \begin{equation}\label{eq6.8}
     W(\lambda, x)
       \begin{pmatrix}
       C'_{0}(x)   \\
       C'_{1}(x)   \\
       \vdots      \\
       C'_{p-2}(x)   \\
       C'_{p-1}(x)
       \end{pmatrix}=
          \begin{pmatrix}
          0   \\
          0   \\
          \vdots \\
          0   \\
          if(x)/\lambda^{p-1}
          \end{pmatrix},
     \end{equation}
   where
     $$
     W(\lambda, x)=
        \begin{pmatrix}
         e^{i\lambda\zeta_{0}x}         &     e^{i\lambda\zeta_{1}x}      & \cdots
         & e^{i\lambda\zeta_{p-1}x} \\
        \zeta_{0}e^{i\lambda\zeta_{0}x} & \zeta_{1}e^{i\lambda\zeta_{1}x} & \cdots
         & \zeta_{p-1}e^{i\lambda\zeta_{p-1}x} \\
          \vdots &  \vdots &  \ddots & \vdots \\
        \zeta_{0}^{p-1}e^{i\lambda\zeta_{0}x} & \zeta_{1}^{p-1}e^{i\lambda\zeta_{1}x} & \cdots
         & \zeta_{p-1}^{p-1}e^{i\lambda\zeta_{p-1}x}
        \end{pmatrix}.
     $$
   Since
     $$
     W(\lambda, x)^{-1}=\frac{1}{p}
        \begin{pmatrix}
         e^{-i\lambda\zeta_{0}x}   & e^{-i\lambda\zeta_{0}x}/\zeta_{0}      & \cdots
         &e^{-i\lambda\zeta_{0}x}/\zeta_{0}^{p-1}  \\
        e^{-i\lambda\zeta_{1}x} &e^{-i\lambda\zeta_{1}x}/\zeta_{1} & \cdots
         &e^{-i\lambda\zeta_{1}x}/\zeta_{1}^{p-1} \\
          \vdots &  \vdots &  \ddots & \vdots \\
        e^{-i\lambda\zeta_{p-1}x} & e^{-i\lambda\zeta_{p-1}x}/\zeta_{p-1} & \cdots
         &e^{-i\lambda\zeta_{p-1}x}/\zeta_{p-1}^{p-1}
        \end{pmatrix},
     $$
   then
     $$
     \begin{pmatrix}
       C'_{0}(x)   \\
       C'_{1}(x)   \\
       \vdots      \\
       C'_{p-2}(x)   \\
       C'_{p-1}(x)
       \end{pmatrix}=W(\lambda, x)^{-1}
         \begin{pmatrix}
          0   \\
          0   \\
          \vdots \\
          0   \\
          if(x)/\lambda^{p-1}
          \end{pmatrix}=i\frac{f(x)}{p\lambda^{p-1}}
             \begin{pmatrix}
              e^{-i\lambda\zeta_{0}x}/\zeta_{0}^{p-1} \\
              e^{-i\lambda\zeta_{1}x}/\zeta_{1}^{p-1} \\
              \vdots      \\
              e^{-i\lambda\zeta_{p-2}x}/\zeta_{p-2}^{p-1} \\
              e^{-i\lambda\zeta_{p-1}x}/\zeta_{p-1}^{p-1}
             \end{pmatrix}.
     $$
   Therefore, for $k=0,1,\ldots,p-1$ we have
     $$ C_{k}(x)=\frac{i}{\lambda^{p-1}}\int_{0}^{x}\frac{e^{-i\lambda\zeta_{k}t}}{p\zeta_{k}^{p-1}}
     f(t)dt +D_{k},\quad (D_{k}=\text{const}).
     $$
   Then
     $$
     y_{1}(\lambda,x)=\frac{i}{\lambda^{p-1}}\int_{0}^{x}\frac{1}{p}\sum_{k=0}^{p-1}
       \frac{1}{\zeta_{k}^{p-1}}e^{i\lambda\zeta_{k}(x-t)}f(t)dt+\sum_{k=0}^{p-1}
       D_{k}e^{i\lambda\zeta_{k}x}.
     $$
   From conditions $y_{1}(\lambda,0)=y'_{1}(\lambda,0)=\cdots=y^{(p-1)}_{1}(\lambda,0)=0$ we obtain $D_{k}=0$ for all $k$. Since
     $$
     \frac{1}{p}\sum_{k=0}^{p-1}
       \frac{1}{\zeta_{k}^{p-1}}e^{i\lambda\zeta_{k}(x-t)}=s_{p-1}(i\lambda(x-t)),
     $$
   then
     \begin{equation}\label{eq6.9}
     y_{1}(\lambda,x)=\frac{i}{\lambda^{p-1}}\int_{0}^{x}s_{p-1}(i\lambda(x-t))f(t)dt.
     \end{equation}
   As a result we have
   \begin{lemma}\label{lem6.1}
   The solution of the Cauchy problem \emph{(\ref{eq6.1}), (\ref{eq6.2})} is the function
     \begin{equation}\label{eq6.10}
     y(\lambda,x)=\sum_{k=0}^{p-1}y_{k}\frac{1}{(i\lambda)^{k}}s_{k}(i\lambda x)+\frac{i}{\lambda^{p-1}}\int_{0}^{x}s_{p-1}(i\lambda(x-t))f(t)dt.
     \end{equation}
   \end{lemma}

  \section{Self-adjointness}\label{s:7}

  Let $M$ be the set of $p$ times differentiable functions from $L^{2}(0,l)$ $(0<l<\infty)$. Use the Lagrange formula for differential operation $(-iD)^{p}$ acting on $M$.
    \begin{multline}\label{eq7.1}
    \int_{0}^{l}(-iD)^{p}y(x)\overline{z(x)}dx=\frac{1}{i^{p}}[Q_{l}(y,z)-Q_{0}(y,z)]+\\
      \int_{0}^{l}y(x)\overline{(-iD)^{p}z(x)}dx,\quad (y,z\in M),
    \end{multline}
  where
    \begin{multline}\label{eq7.2}
    Q_{x}(y,z)=\\
    y^{(p-1)}(x)\overline{z(x)}-y^{(p-2)}(x)\overline{z(x)}'+\cdots+(-1)^{p-1}y(x)
    \overline{z(x)}^{(p-1)},
    \end{multline}
  here $Q_{l}(y,z)=Q_{x}(y,z)|_{x=l}$,\; $Q_{0}(y,z)=Q_{x}(y,z)|_{x=0}$.

  The matrix $J_{p}$ of the bilinear form (\ref{eq7.2}) has the size $p\times p$
    \begin{equation}\label{eq7.3}
    J_{p}=
       \begin{pmatrix}
            0     &      0     & \cdots &    0   & 1  \\
            0     &      0     & \cdots &   -1   & 0  \\
         \vdots   &   \vdots   & \ddots & \vdots & \vdots \\
            0     & (-1)^{p-2} & \cdots &   0    &    0   \\
       (-1)^{p-1} &      0     & \cdots &   0    &    0
       \end{pmatrix},
    \end{equation}
  and satisfies the conditions
    \begin{equation}\label{eq7.4}
    J_{p}^{\ast}=(-1)^{p-1}J_{p},\quad J_{p}^{2}=I\;\,(p=2k+1),\quad
       J_{p}^{2}=-I\;\,(p=2k).
    \end{equation}

  \textbf{Case} $p=2k$.\;
  Since $(iJ_{p})^{\ast}=iJ_{p}$ and $(iJ_{p})^{2}=I$, the eigenvalues of the matrix $iJ_{p}$ are $\pm 1$. Therefore, the eigenvalues of  $J_{p}$ are $\pm i$.

  Orthonormal eigenvectors
    \begin{equation}\label{eq7.5}
      \begin{array}{c}
      e_{1}^{+}=\frac{1}{\sqrt 2}\text{col}\,(1,0,\ldots,0,0,\ldots,0,i) \\
      e_{2}^{+}=\frac{1}{\sqrt 2}\text{col}\,(0,1,\ldots,0,0,\ldots,-i,0) \\
      \cdots\\
      e_{k}^{+}=\frac{1}{\sqrt 2}\text{col}\,(0,0,\ldots,1,(-1)^{k-1}i,\ldots,0,0)
      \end{array}
    \end{equation}
   form  eigen subspace $E_{+}^{(k)}$ $(\dim E_{+}^{(k)}=k)$ of the operator $J_{p}=J_{2k}$ corresponding to the eigenvalue $\lambda=+i$, $J_{2k}e_{s}^{{+}}=ie_{s}^{{+}}$ $(1\leq s\leq k)$.

   Likewise, the orthonormal eigenvectors
    \begin{equation}\label{eq7.6}
      \begin{array}{c}
      e_{1}^{-}=\frac{1}{\sqrt 2}\text{col}\,(-1,0,\ldots,0,0,\ldots,0,i) \\
      e_{2}^{-}=\frac{1}{\sqrt 2}\text{col}\,(0,-1,\ldots,0,0,\ldots,-i,0) \\
      \cdots\\
      e_{k}^{-}=\frac{1}{\sqrt 2}\text{col}\,(0,0,\ldots,-1,(-1)^{k-1}i,\ldots,0,0)
      \end{array}
    \end{equation}
   form  eigen subspace $E_{-}^{(k)}$ $(\dim E_{-}^{(k)}=k)$ corresponding to the eigenvalue $\lambda=-i$.

   Therefore,
    \begin{equation}\label{eq7.7}
      E^{p}=E^{p}_{+}\oplus E^{p}_{-}.
    \end{equation}
   Any vector $h\in E^{p}$ has a unique decomposition
      \begin{equation}\label{eq7.8}
      h=h_{+}+h_{-}, \quad h_{\pm}=P_{\pm}h,
      \end{equation}
   where $P_{\pm}$ are orthoprojectors on $E_{\pm}$. Therefore,
      $$
      P_{+}h=\sum_{s=1}^{k}h_{s}^{+}e_{s}^{+},\quad P_{-}h=\sum_{s=1}^{k}h_{s}^{-}e_{s}^{-},
      $$
   where $e_{s}^{\pm}$ are the vectors (\ref{eq7.5}), (\ref{eq7.6}), and
     \begin{equation}\label{eq7.9}
     h_{s}^{\pm}=\langle h,e_{s}^{\pm}\rangle=\frac{\pm h_{s}+(-1)^{p-s}ih_{p-s+1}}{\sqrt 2},
     \quad l=\text{col}\,(h_{1},\ldots,h_{p}).
     \end{equation}
   From (\ref{eq7.8}) it follows that the quadratic form (\ref{eq7.8}) has the form
     \begin{equation}\label{eq7.10}
     \langle J_{p}h,h\rangle=i\{\langle h_{+},h_{+}\rangle-\langle h_{-},h_{-}\rangle\}.
     \end{equation}
   Equality (\ref{eq7.10}) allows us to distinguish two types of boundary conditions  which ensure the self-adjointness of the operator $(-iD)^{p}$ for an even $p=2k$.

   The first type of boundary conditions is determined by those $h$ that annul the quadratic form
   $\langle J_{p}h,h\rangle$. Let
     \begin{equation}\label{eq7.11}
     h_{V}=h_{+}+Vh_{+},\quad (h_{-}=Vh_{+}),
     \end{equation}
   where $V: E_{+}\rightarrow E_{-}$ is an arbitrary unitary operator. Given (\ref{eq7.10}), we get $\langle J_{p}h_{V},h_{V}\rangle=0$. The operator $V$ has the form $V=V_{+}^{\ast}V_{0}V_{-}$, where $V_{\pm}: E_{\pm}^{(k)}\rightarrow E^{(k)}$ are unitary mappings transforming  $\{e_{s}^{\pm}\}_{s=1}^{k}$ into the standard basis $\{e_{s}\}_{s=1}^{k}$ of the space $E^{(k)}$ $(\dim E^{(k)}=k)$ and $V_{0}$ is a unitary operator in $E^{(k)}$.
   Since $h_{-}=Vh_{+}$, then from (\ref{eq7.9}) it follows
     \begin{equation}\label{eq7.12}
     \widetilde{h}_{1}+iJ_{k}\widehat{h}_{1}=V_{0}(-\widetilde{h}_{1}+iJ_{k}\widehat{h}_{1}),
     \end{equation}
   where $J_{k}$ is (\ref{eq7.3})  for $p=k$ and $\widetilde{h}_{1},\widehat{h}_{1}\in E^{(k)}$ are vectors
     \begin{equation}\label{eq7.13}
     \widetilde{h}_{1}=\text{col}\,(h_{1},\ldots,h_{k}),\;
     \widehat{h}_{1}=\text{col}\,(h_{k+1},\ldots,h_{2k}),\;
     (h=\widetilde{h}_{1}\oplus\widehat{h}_{1}).
     \end{equation}
   Rewrite (\ref{eq7.12}) as
     $$
     i(I-V_{0})J_{k}\widehat{h}_{1}=-(I+V_{0})\widetilde{h}_{1},
     $$
   or
     \begin{equation}\label{eq7,14}
     J_{k}\widehat{h}_{1}=B_{0}\widetilde{h}_{1},
     \end{equation}
   where
     \begin{equation}\label{eq7,15}
     B_{0}=i(I-V_{0})^{-1}(I+V_{0})
     \end{equation}
   Since $V_{0}V_{0}^{\ast}=I$, then $B_{0}^{\ast}=B_{0}$.
   \medskip

   Introduce the notations
     \begin{equation}\label{eq7.16}
     \begin{array}{c}
     \;Y(x)=\text{col}\,(y(x),\ldots,y^{(k-1)}(x),y^{(k)}(x)\ldots,y^{(p-1)}(x)),\\
     Y_{0}(x)=\text{col}\,(y(x),y'(x),\ldots,y^{(k-1)}(x))\in E^{(k)},\qquad\qquad\\
     Y_{1}(x)=\text{col}\,(y^{(k)}(x),y^{(k+1)}(x),\ldots,y^{(p-1)}(x))\in E^{(k)},\quad\\
     Y(x)=Y_{0}(x)\oplus Y_{1}(x)\in E^{(2k)}.\qquad
     \end{array}
     \end{equation}
   \begin{proposition}\label{pro7.1}
   Let $p=2k$, $Y_{0}(x)$, $Y_{1}(x)$ and $J_{k}$ be determined by the relations \emph{(\ref{eq7.16})} and \emph{(\ref{eq7.3})}. If $B_{0}$, $B_{l}$ are arbitrary self-adjoint $k\times k$-matrices, then
   on the set of functions from $L^{2}(0,l)$, $0<l< \infty$ that satisfy the boundary conditions
      \begin{equation}\label{eq7.17}
      J_{k}Y_{1}(0)=B_{0}Y_{0}(0),\quad J_{k}Y_{1}(l)=B_{l}Y_{0}(l),
      \end{equation}
   the operator $(-iD)^{p}$ is self-adjoint.
   \end{proposition}
   Boundary conditions (\ref{eq7.17}) refer to \emph{separated boundary conditions}.
   \medskip

   The second type of self-adjoint boundary conditions is obtained as follows. The non-integral terms of the Lagrange formula (\ref{eq7.1}1) have the form
     \begin{multline}\nonumber
     Q_{l}(y,z)-Q_{0}(y,z)=\langle J_{p}Y(l),Z(l)\rangle-\langle J_{p}Y(0),Z(0)\rangle=\\
       \langle Y_{+}(l),Z_{+}(l)\rangle-\langle Y_{-}(l),Z_{-}(l)\rangle-
       \langle Y_{+}(0),Z_{+}(0)\rangle+\langle Y_{-}(0),Z_{-}(0)\rangle.
     \end{multline}
   This expression is zero if
    \begin{equation}\label{eq7.18}
      \begin{array}{c}
      (1)\quad \langle Y_{+}(l),Z_{+}(l)\rangle=\langle Y_{+}(0),Z_{+}(0)\rangle,\\
      (2)\quad \langle Y_{-}(l),Z_{-}(l)\rangle=\langle Y_{-}(0),Z_{-}(0)\rangle.
      \end{array}
    \end{equation}
   Equalities (\ref{eq7.18}) are satisfied when
      \begin{equation}\nonumber
      \begin{array}{c}
       Y_{+}(l)=VY_{+}(0)\quad (Z_{+}(l)=VZ_{+}(0)),\\
       Y_{-}(l)=\widetilde{V}Y_{-}(0)\quad (Z_{-}(l)=\widetilde{V}Z_{-}(0)),
      \end{array}
      \end{equation}
   where $V$ and $\widetilde{V}$ are unitary operators. Using (\ref{eq7.12}) and (\ref{eq7.9}), we rewrite these equalities in the form
     \begin{equation}\nonumber
      \begin{array}{c}
       Y_{0}(l)+iJ_{k}Y_{1}(j)=V(Y_{0}(0)+iJ_{k}Y_{1}(0)),\\
       -Y_{0}(l)+iJ_{k}Y_{1}(j)=\widetilde{V}(-Y_{0}(0)+iJ_{k}Y_{1}(0)).
      \end{array}
      \end{equation}

   \begin{proposition}\label{pro7.2}
   Let $p=2k$, $Y_{0}(x)$, $Y_{1}(x)$ and $J_{k}$ be determined by the relations \emph{(\ref{eq7.16})} and \emph{(\ref{eq7.3})}. If $V$, $\widetilde{V}$ are arbitrary unitary operators in $E^{(k)}$, then on the set of functions from $L^{2}(0,l)$, $0<l< \infty$ that satisfy the boundary conditions
      \begin{equation}\label{eq7.19}
       \begin{array}{c}
       2Y_{0}(l)=(V+\widetilde{V})Y_{0}(0)+i(V-\widetilde{V})J_{k}Y_{1}(0),\\
       2iJ_{k}Y_{1}(l)=(V-\widetilde{V})Y_{0}(0)+i(V+\widetilde{V})J_{k}Y_{1}(0),
       \end{array}
      \end{equation}
   the operator $(-iD)^{p}$ is self-adjoint.
   \end{proposition}

   \begin{remark}\label{rem7.1}
   If $V=\widetilde{V}$, then equalities (\ref{eq7.19}) have the form
      \begin{equation}\label{eq7.20}
      Y_{0}(l)=VY_{0}(0),\quad J_{k}Y_{1}(l)=VJ_{k}Y_{1}(0),
      \end{equation}
   and if $V=-\widetilde{V}$, then
      \begin{equation}\label{eq7.21}
      Y_{0}(l)=iVJ_{k}Y_{1}(0),\quad J_{k}Y_{1}(l)=-iVY_{0}(0).
      \end{equation}
   \end{remark}

   Boundary conditions (\ref{eq7.19}) (as well as (\ref{eq7.20}) and (\ref{eq7.21})) refer to the so-called \emph{unseparated self-adjoint boundary conditions}.
   \medskip

   \textbf{Case} $p=2k+1$.\; From (\ref{eq7.3}), it follows that $J_{p}^{\ast}=J_{p}$ and $J_{p}^{2}=I$. Therefore, the eigenvalues of $J_{p}$ are equal to $\pm 1$.   Orthonormal eigenvectors
    \begin{equation}\label{eq7.22}
      \begin{array}{c}
      \widetilde{e_{1}}^{+}=\frac{1}{\sqrt 2}\text{col}\,(1,0,\ldots,0,0,\ldots,0,1) \\
      \widetilde{e_{2}}^{+}=\frac{1}{\sqrt 2}\text{col}\,(0,1,\ldots,0,0,\ldots,-1,0) \\
      \cdots\\
      \widetilde{e_{k}}^{+}=\frac{1}{\sqrt 2}\text{col}\,(0,0,\ldots,1,0,(-1)^{k-1},\ldots,0,0)
      \end{array}
    \end{equation}
   form  eigen subspace $\widetilde{E}_{+}^{(k)}$ $(\dim \widetilde{E}_{+}^{(k)}=k)$ of the operator $J_{p}=J_{2k+1}$ corresponding to the eigenvalue $\lambda=+1$, $J_{2k+1}\widetilde{e_{s}}^{{+}}=1\cdot \widetilde{e_{s}}^{{+}}$ $(1\leq s\leq k)$. Likewise, the orthonormal eigenvectors
      \begin{equation}\label{eq7.23}
      \begin{array}{c}
      \widetilde{e_{1}}^{-}=\frac{1}{\sqrt 2}\text{col}\,(-1,0,\ldots,0,0,\ldots,0,1) \\
      \widetilde{e_{2}}^{-}=\frac{1}{\sqrt 2}\text{col}\,(0,-1,\ldots,0,0,\ldots,-1,0) \\
      \cdots\\
      \widetilde{e_{k}}^{-}=\frac{1}{\sqrt 2}\text{col}\,(0,0,\ldots,-1,0,(-1)^{k-1},\ldots,0,0)
      \end{array}
    \end{equation}
   form  eigen subspace $\widetilde{E}_{-}^{(k)}$ $(\dim \widetilde{E}_{-}^{(k)}=k)$ of the operator $J_{p}=J_{2k+1}$ corresponding to the eigenvalue $\lambda=-1$, $J_{2k+1}\widetilde{e_{s}}^{-}=-1\cdot \widetilde{e_{s}}^{-}$ $(1\leq s\leq k)$. In addition, the operator $J_{2k+1}$ has an eigenvector
      \begin{equation}\label{eq7.24}
      \widetilde{e}_{k+1}=\frac{1}{\sqrt 2}\text{col}\,(0,0,\ldots,0,1,0,\ldots,0,0)
      \end{equation}
   and $J_{p}\widetilde{e}_{k+1}=(-1)^{k}\widetilde{e}_{k+1}$.

   The space $E^{(p)}=E^{(2k+1)}$ is the sum of three orthogonal subspaces
      $$
      E^{(p)}=\widetilde{E}_{+}^{(k)}\oplus\widetilde{E}_{-}^{(k)}\oplus\{\mu\widetilde{e}_{k+1}\}.
      $$
   Therefore, any vector $h\in E^{(p)}$ has unique decomposition
      \begin{equation}\label{eq7.25}
      h=h_{+}+h_{-}+\widetilde{h}\quad (h_{\pm}=P_{\pm}h,\;\widetilde{h}\in\{\mu\widetilde{e}_{k+1}\}),
      \end{equation}
   where $P_{\pm}$ are orthoprojectors onto subspaces $\widetilde{E}_{\pm}$ and
       \begin{equation}\label{eq7.26}
       h_{\pm}=\sum_{s=1}^{k}\widetilde{h_{s}}^{\pm}\widetilde{e_{s}}^{\pm},\quad
          \widetilde{h}=\mu\widetilde{e}_{k+1},
       \end{equation}
   where
       \begin{equation}\label{eq7.27}
       \widetilde{h_{s}}^{\pm}=\frac{\pm h_{s}+(-1)^{p-s}h_{p-s+1}}{\sqrt 2}\quad (1\leq s\leq k).
       \end{equation}
   Taking into account (\ref{eq7.25}), we obtain
       \begin{equation}\label{eq7.28}
       \langle J_{p}h,h\rangle=\langle \widetilde{h}_{+},\widetilde{h}_{+}\rangle-
            \langle \widetilde{h}_{-},\widetilde{h}_{-}\rangle+(-1)^{k}|h_{k+1}|^{2}.
       \end{equation}

   Similarly to (\ref{eq7.16}), we introduce the notation
     \begin{equation}\label{eq7.29}
     \begin{array}{c}
     Y(x)=\text{col}\,(y(x),\ldots,y^{(k-1)}(x),y^{(k)}(x),y^{(k+1)}(x),\ldots,y^{(p-1)}(x)),\\
     Y_{0}(x)=\text{col}\,(y(x),y'(x),\ldots,y^{(k-1)}(x))\in E^{(k)},\qquad\qquad\\
     Y_{1}(x)=\text{col}\,(y^{(k+1)}(x),y^{(k+2)}(x),\ldots,y^{(p-1)}(x))\in E^{(k)},\,\\
     \end{array}
     \end{equation}
    \begin{proposition}\label{pro7.3}
    Let $p=2k+1$, $Y_{0}(x)$, $Y_{1}(x)$ and $J_{k}$ be determined by the relations \emph{(\ref{eq7.29})} and \emph{(\ref{eq7.3})}. If $\theta\in \mathbb{T}=\{z\in\mathbb{C}\mid |z|=1\}$ and $B_{0}$, $B_{l}$ are arbitrary self-adjoint $k\times k$-matrices, then
   on the set of functions from $L^{2}(0,l)$, $0<l< \infty$ that satisfy the boundary conditions
     \begin{equation}\label{eq7.30}
     y^{(k)}(l)=\theta y^{(k)}(0),\quad J_{k}Y_{1}(0)=iB_{0}Y_{0}(0),\quad
          J_{k}Y_{1}(l)=iB_{1}Y_{0}(l),
     \end{equation}
   the operator $(-iD)^{p}$ is self-adjoint.
    \end{proposition}

   \begin{proof}
   From (\ref{eq7.28}) it follows that
     \begin{multline}\label{eq7.31}
     Q_{l}(y,z)-Q_{0}(y,z)=\langle J_{p}Y(l),Z(l)\rangle-\langle J_{p}Y(0),Z(0)\rangle=\\
     \langle Y_{+}(l),Z_{+}(l)\rangle-\langle Y_{-}(l),Z_{-}(l)\rangle+(-1)^{k}y^{(k)}(l)z^{(k)}(l)-\quad\\
      \langle Y_{+}(0),Z_{+}(0)\rangle-\langle Y_{-}(0),Z_{-}(0)\rangle+(-1)^{k}y^{(k)}(0)z^{(k)}(0).
     \end{multline}
   This expression vanishes if
     $$
     \langle Y_{+}(l),Z_{+}(l)\rangle=\langle Y_{-}(l),Z_{-}(l)\rangle,
     \quad y^{(k)}(l)=\theta y^{(k)}(0),\quad (\theta\in\mathbb{T})
     $$
   and $\langle Y_{+}(0),Z_{+}(0)\rangle=\langle Y_{-}(0),Z_{-}(0)\rangle$.

   Transform, for example, the first relation. We have
     \begin{equation}\label{eq7.32}
     \langle h_{+},h_{+}\rangle-\langle h_{-},h_{-}\rangle=0.
     \end{equation}
   Obviously, this relation is satisfied for vectors of the form
     \begin{equation}\label{eq7.33}
     h=h_{+}+Vh_{+},
     \end{equation}
   where $V:\widetilde{E}_{+}^{(k)}\rightarrow \widetilde{E}_{-}^{(k)}$ is an arbitrary unitary operator. The equality $h_{-}=Vh_{+}$ means that (see (\ref{eq7.12}))
     $$
     \widetilde{h}_{1}+J_{k}\widehat{h}_{1}=V_{0}(-\widetilde{h}_{1}+J_{k}\widehat{h}_{1}),
     $$
   where $J_{k}$ has the form (\ref{eq7.3}) and
     $$
     \widetilde{h}_{1}=\text{col}\,(h_{1},\ldots,h_{k}),\quad
     \widehat{h}_{1}=\text{col}\,(h_{k+2},\ldots,h_{2k+1}).
     $$
   This implies
     $$
     (I-V_{0})J_{k}\widehat{h}_{1}=-(I+V_{0})\widetilde{h}_{1}
     $$
   or
     $$
     J_{k}\widehat{h}_{1}=iB_{0}\widetilde{h}_{1},
     $$
   where $B_{0}=i(I-V_{0})^{-1}(I+V_{0})$ is a self-adjoint operator.
   \end{proof}

   Relations (\ref{eq7.30}) are analogous to the separated boundary conditions for $p=2k+1$.
   \medskip

   If in equality (\ref{eq7.31}) we put
     $$
     \langle Y_{+}(l),Z_{+}(l)\rangle=\langle Y_{+}(0),Z_{+}(0)\rangle,
     \quad y^{(k)}(l)=\theta y^{(k)}(0),\quad (\theta\in\mathbb{T})
     $$
   and $\langle Y_{-}(l),Z_{-}(l)\rangle=\langle Y_{-}(0),Z_{-}(0)\rangle$, then $Q_{l}(y,z)-Q_{0}(y,z)=0$. Then, similarly to Proposition \ref{pro7.2}, we obtain the following statement.

     \begin{proposition}\label{eq7.4}
      Let $p=2k+1$, $Y_{0}(x)$, $Y_{1}(x)$ and $J_{k}$ be determined by the relations \emph{(\ref{eq7.29})} and \emph{(\ref{eq7.3})}. If $V$ and $\widetilde{V}$ are arbitrary unitary operators in $E^{(k)}$, then on the set of functions from $L^{2}(0,l)$, $0<l< \infty$ that satisfy the boundary conditions
     \begin{equation}\label{eq7.34}
       \begin{array}{c}
       y^{(k)}(l)=\theta y^{(k)}(0)\quad (\theta\in \mathbb{T}=\{z\in\mathbb{C}\mid |z|=1\}),\\
       2Y_{0}(l)=(V+\widetilde{V})Y_{0}(0)+(V-\widetilde{V})J_{k}Y_{1}(0),\\
       2J_{k}Y_{1}(l)=(V-\widetilde{V})Y_{0}(0)+(V+\widetilde{V})J_{k}Y_{1}(0),
       \end{array}
     \end{equation}
   the operator $(-iD)^{p}$ is self-adjoint.
     \end{proposition}

     \begin{remark}\label{rem7.2}
     If $\widetilde{V}=V$, then conditions (\ref{eq7.34}) take the form
       $$
       Y_{0}(l)=VY_{0}(0), \quad J_{k}Y_{1}(l)=VJ_{k}Y_{1}(0),\quad y^{(k)}(l)=\theta y^{(k)}(0)\; (\theta\in \mathbb{T}).
       $$
     For $\widetilde{V}=-V$ we obtain
       $$
       Y_{0}(l)=VJ_{k}Y_{1}(0), \quad J_{k}Y_{1}(l)=-VY_{0}(0),\quad y^{(k)}(l)=\theta y^{(k)}(0)\; (\theta\in \mathbb{T}).
       $$
     \end{remark}

   \section{\large{Orthonormal bases}}\label{s:8}

   Let $W_{2}^{3}(0,l)$ be a linear manifold of thrice differentiable functions $y(x)$ from $L^{2}(0,l)$ $(0<l<\infty)$ satisfying the boundary conditions
      \begin{equation}\label{eq8.1}
      y(0)=0,\quad y'(0)=\theta y'(l),\quad y(l)=0,\quad (\theta\in \mathbb{T}).
      \end{equation}
   By Proposition \ref{pro7.3}, $L_{\theta}=(-iD)^{3}$ on $W_{2}^{3}(0,l)$ is a self-adjoint operator.

   Find the eigenfunctions of the operator $L_{\theta}$. The solution of the equation
      \begin{equation}\label{eq8.2}
      L_{\theta}y(\lambda,x)=\lambda^{3}y(\lambda,x)
      \end{equation}
   that satisfies the first boundary condition (\ref{eq8.1}) $(y(0)=0)$ has the form
      \begin{equation}\label{eq8.3}
       y(\lambda,x)=a_{1}\frac{s_{1}(i\lambda x)}{i\lambda}+
            a_{2}\frac{s_{2}(i\lambda x)}{(i\lambda)^{2}}.
      \end{equation}
   Since
     $$
     y'(\lambda,x)=a_{1}s_{0}(i\lambda x)+
            a_{2}\frac{s_{1}(i\lambda x)}{i\lambda},
     $$
   it follows from the boundary condition $y'(0)=\theta y'(l)$ that
     \begin{equation}\label{eq8.4}
     a_{1}=a_{1}\theta s_{0}(i\lambda l)+
            a_{2}\theta\frac{s_{1}(i\lambda l)}{i\lambda}.
     \end{equation}
   \begin{remark}\label{rem8.1}
   Note that $\lambda=0$ is not an eigenvalue of $L_{\theta}$ for $\theta\neq -1$. Indeed, if $\lambda=0$, then $y(0,x)=y_{0}+y_{1}x+y_{2}x^{2}$ follows from (\ref{eq8.2}). Since $y(0)=0$, then $y_{0}=0$. Considering $y(l)=0$, we get $l(y_{1}+y_{1}l)=0$. Therefore, $y(0,x)=y_{2}(x^{2}-xl)$. Since $y'(0)=\theta y'(l)$, then $-y_{2}l=\theta y_{2}l$. From here $\theta=-1$. If $\theta=-1$, then $y(\lambda,x)=C(x^{2}-xl)$ $(C\neq 0)$ is the eigenfunction of the operator $L_{-1}$, corresponding to the eigenvalue $\lambda=0$.
    \end{remark}

   \emph{In the future, we assume that $\theta\neq -1$}.
   \medskip

   From (\ref{eq8.4}) we get
     $$
     a_{1}=a_{2}\frac{\theta s_{1}(i\lambda l)}{i\lambda (1-\theta s_{0}(i\lambda l))}.
     $$
   Therefore, from (\ref{eq8.3}) we obtain
     \begin{equation}\label{eq8.5}
     y(\lambda,x)=\frac{a_{2}\left[\theta s_{1}(i\lambda x)s_{1}(i\lambda l)+s_{2}(i\lambda x)(1-\theta s_{0}(i\lambda l))\right]}{(i\lambda)^{2}(1-\theta s_{0}(i\lambda l))}
     \end{equation}
   The third boundary condition $y(l)=0$ takes the form
     \begin{equation}\label{eq8.6}
     \frac{\theta s_{1}(i\lambda l)^{2}+s_{2}(i\lambda l)(1-\theta s_{0}(i\lambda l))}{1-\theta s_{0}(i\lambda l)}=0
     \end{equation}

   \begin{lemma}\label{lem8.1}
   If\; $1-\theta s_{0}(i\lambda l)=0$\; $(|\theta|=1)$, then the boundary condition \emph{(\ref{eq8.6})} is not satisfied.
   \end{lemma}

   \begin{proof}
   If  (\ref{eq8.6}) is true for $s_{0}(i\lambda l)=\overline{\theta}$, then $s_{1}(i\lambda l)=0$. From the main identity (\ref{eq4.5}) we get
     $$
     \overline{\theta}^{3}+s_{2}(i\lambda l)^{3}=1\quad\Rightarrow\quad
     s_{2}(i\lambda l)=(1-\overline{\theta}^{3})^{1/3}.
     $$
   According to the Euler formula (\ref{eq4.4}) we get
     $$
     e^{i\lambda l}=\overline{\theta}+(1-\overline{\theta}^{3})^{1/3}.
     $$
   Since $e^{i\lambda l}e^{-i\lambda l}=1$, then
     $$
     \frac{\theta}{(1-\theta^{3})^{1/3}}+\frac{\overline{\theta}}{(1-\overline{\theta}^{3})^{1/3}}=-1.
     $$
   Therefore, there exists $b\in\mathbb{R}$ such that
     \begin{equation}\label{eq8.7}
     \frac{\theta}{(1-\theta^{3})^{1/3}}=-\frac{1}{2}+i\frac{b}{2}.
     \end{equation}
   Show that (\ref{eq8.7}) is impossible for $|\theta|=1$. Introducing the notation
      $$
      \left(-\frac{1}{2}+i\frac{b}{2}\right)^{3}=\frac{-1+3b^{2}}{8}+i\frac{3b-b^{3}}{8}=z=x+iy,
      $$
   from (\ref{eq8.7}) we obtain
      $$
      \theta^{3}=z(1-\theta^{3})\quad\Rightarrow\quad \theta^{-3}=
         \left(1+\frac{x}{x^{2}+y^{2}}\right)-i\frac{y}{x^{2}+y^{2}}.
      $$
   Condition $|\theta^{-3}|=1$ is satisfied only when $x=-\frac{1}{2}$. Therefore, for $b$ we obtain the equation
      $$
      \frac{-1+3b^{2}}{8}=-\frac{1}{2}\quad\Leftrightarrow\quad b^{2}+1=0,
      $$
   which has no real solutions. Contradiction.
   \end{proof}

   We call the function
      \begin{equation}\label{eq8.8}
      \Delta_{\theta}(\lambda)=\frac{1}{\lambda^{2}}
         \left[\theta s_{1}(i\lambda l)^{2}+s_{2}(i\lambda l)(1-\theta s_{0}(i\lambda l))\right]
      \end{equation}
   the characteristic function of the operator $L_{\theta}$.

   \begin{remark}\label{rem8.1}
   From the second and third boundary conditions (\ref{eq8.1}) for the function $y(\lambda,x)$ (\ref{eq8.3}) we obtain a system of equations
     \begin{equation}\label{eq8.9}
       \begin{cases}
       a_{1}(1-\theta s_{0}(i\lambda l))-a_{2}\theta s_{1}(i\lambda l)/i\lambda=0\\
       a_{1}s_{1}(i\lambda l)/i\lambda+a_{2}s_{2}(i\lambda l)/(i\lambda)^{2}=0
       \end{cases}
     \end{equation}
   for $a_{1}$ and $a_{2}$. The existence of nontrivial solutions for this system is equivalent to the vanishing of its determinant, that is, the characteristic function (\ref{eq8.8}).
   \end{remark}

   Taking into account (\ref{eq4.11}), we rewrite the characteristic function in the form
     \begin{equation}\label{eq8.10}
     \Delta_{\theta}(\lambda)=\frac{1}{\lambda^{2}}\left[s_{2}(i\lambda l)+\theta s_{2}(-i\lambda l)\right].
     \end{equation}

   \begin{lemma}\label{lem8.2}
   The characteristic function \emph{(\ref{eq8.10})} is an entire function of exponential type that satisfies the conditions
    \begin{equation}\label{eq8.11}
     \begin{array}{ccc}
     (\emph{i}) & &\overline{\Delta_{\theta}(\lambda})=\Delta_{\overline{\theta}}(-\overline{\lambda})=
          \overline{\theta}\Delta_{\theta}(\overline{\lambda}),\qquad\\
     (\emph{ii}) & & \Delta_{\theta}(-\lambda)=\theta\Delta_{\overline{\theta}}(\lambda),\qquad\qquad\qquad\\
     (\emph{iii}) & & \Delta_{\theta}(\lambda\zeta_{j})=\Delta_{\theta}(\lambda),\qquad
     j=0,1,2.
     \end{array}
    \end{equation}
   If $\theta\neq -1$, then zeros of the function $\Delta_{\theta}(i\lambda l)$ are simple and have the form
     \begin{equation}\label{eq8.12}
     \mathcal{Z}_{0}=\{\;\zeta_{1}^{s}\lambda_{n}(\varphi),\quad -\zeta_{1}^{s}\lambda_{n}(-\varphi),\quad s=0,1,2,\;\;n\in\mathbb{N}\},
     \end{equation}
   where $\varphi=\arg\theta$ $(|\varphi|<\pi)$, $\{\lambda_{n}(\varphi)\}_{0}^{\infty}$
   $\left(\{-\lambda_{n}(-\varphi)\}_{0}^{\infty}\right)$ are the positive \emph{(}negative\emph{)} zeros of $\Delta_{\theta}(\lambda)$, numbered in ascending \emph{(}descending\emph{)} order.

   Zeros $\lambda_{n}(\varphi)$ and $-\lambda_{n}(-\varphi)$ belong to the interval
   $$
   \left(\frac{\pi}{2}(2n-1)-\varphi,\frac{\pi}{2}(2n+1)-\varphi\right),
   \left(-\frac{\pi}{2}(2n+1)+\varphi,-\frac{\pi}{2}(2n-1)+\varphi\right).
   $$
   The following asymptotics exist,
    \begin{multline}\label{eq8.13}
    \quad\;\lambda_{n}(\varphi)=\frac{2\pi n}{l}-\varphi-\frac{\pi}{3l}+o\left(\frac{1}{n}\right)\quad(n\rightarrow\infty)\\
     -\lambda_{n}(-\varphi)=-\frac{2\pi n}{l}+\varphi+\frac{\pi}{3l}+o\left(\frac{1}{n}\right)\quad(n\rightarrow\infty).\qquad\qquad
     \quad
    \end{multline}
   \end{lemma}

   \begin{proof}
   Identities (\ref{eq8.11}) follow from Definition \ref{def3.1}, Proposition \ref{pro3.2} (v).

   Using the Taylor series (\ref{eq3.7}) it is easy to show that
      $$
      \Delta_{\theta}(\lambda)|_{\lambda=0}=-l^{2}(1+\theta)/2\neq 0\quad\text{for}\quad \theta\neq -1.
      $$
   First, find the zeros lying on the lines $l_{k}=\{z\in\mathbb{C}\mid z=x\zeta_{k},\;x\in\mathbb{R}\}$, $k=0,1,2$. Later we will show that there are no other zeros.

   It follows from (\ref{eq8.11}) (ii), (iii) that it suffices to find the zeros $\lambda_{n}(\varphi)$ of the function $\Delta_{\theta}(\lambda)$ that lie on $\mathbb{R_{+}}=\{x\in\mathbb{R}\mid x>0\}$.

   Using $\theta=e^{i\varphi}$, we represent the characteristic function in the form
     $$
     \Delta_{\theta}(\lambda)=\frac{2e^{i\varphi/2}}{3\lambda^{2}}\sum_{k=0}^{2}\zeta_{k}
        \cos\left(\lambda\zeta_{k}l-\frac{\varphi}{2}\right).
     $$

   Since $\theta\neq -1$, then the equation $\Delta_{\theta}(\lambda)=0$ is equivalent to the equation
     $$
     \sum_{k=0}^{2}\zeta_{k}
        \cos\left(\lambda\zeta_{k}l-\frac{\varphi}{2}\right)=0.
     $$
   From this we get
     \begin{multline}\nonumber
     \cos\left(\lambda l-\frac{\varphi}{2}\right)-
     \cos\left(\frac{\lambda l+\varphi}{2}\right)\cosh\frac{\lambda l\sqrt 3}{2}-\\
     \sqrt 3\sin\left(\frac{\lambda l+\varphi}{2}\right)\sinh\frac{\lambda l\sqrt 3}{2}=0,\qquad\qquad
     \end{multline}
   or
     \begin{multline}\label{eq8.14}
     \cos\left(\frac{\lambda l+\varphi}{2}\right)\left(\cos\left(\frac{\lambda l}{2}-\varphi\right)-\cosh\frac{\lambda l\sqrt 3}{2}\right)=\\
        \sin\left(\frac{\lambda l+\varphi}{2}\right)\left(\sin\left(\frac{\lambda l}{2}-\varphi\right)+\sqrt 3\sinh\frac{\lambda l\sqrt 3}{2}\right).
     \end{multline}
   Note that if $\varphi\in(0,\pi)$, then
     $$
     \Phi(x):=\sin(x-\varphi)+\sqrt 3 \sinh x\sqrt 3>0\quad\text{for}\quad x>0.
     $$
   If $\varphi\in(-\pi,0)$, then $\Phi(x)$ has a simple zero $x_{0}$ lying in $(0,\pi)$. In this case $\Phi(x)<0$, for $x\in(0,x_{0})$, and  $\Phi(x)>0$, for $x\in(x_{0},\infty)$.

   Let $x_{0}$ be zero of $\Phi(x)$ $(\varphi\in(-\pi,0))$. Then (\ref{eq8.14}) implies $\cos(x_{0}+\varphi/2)=0$. Therefore $\lambda_{0}(\varphi)=(\pi-\varphi)/2$ is the zero of the characteristic function $\Delta_{\theta}(\lambda)$. For $\lambda>\lambda_{0}(\varphi)$, the inequality
   $\sin\left(\frac{\lambda l}{2}-\varphi\right)+\sqrt 3 \sinh \frac{\lambda l\sqrt 3}{2}>0$ holds, and (\ref{eq8.14}) is equivalent to the equation
      \begin{equation}\label{eq8.15}
      \tan\left(\frac{\lambda l+\varphi}{2}\right)=f(\lambda,\varphi),
      \end{equation}
   where
      \begin{equation}\label{eq8.16}
      f(\lambda,\varphi)=\frac{\cos\left(\frac{\lambda l}{2}-\varphi\right)-\cosh\left(\frac{\lambda l\sqrt 3}{2}\right)}{\sin\left(\frac{\lambda l}{2}-\varphi\right)+\sqrt 3\sinh\left(\frac{\lambda l\sqrt 3}{2}\right)}\quad (\lambda>\lambda_{0}(\varphi)).
      \end{equation}
   The derivative with respect to $\lambda$ of this function is
     $$
     f'(\lambda,\varphi)=l
       \frac{1-\cos\left(\frac{\lambda l}{2}-\varphi\right)\cosh\frac{\lambda l\sqrt 3}{2}-\sqrt 3\sin\left(\frac{\lambda l}{2}-\varphi\right)\sinh\frac{\lambda l\sqrt 3}{2}}
       {\left(\sin\left(\frac{\lambda l}{2}-\varphi\right)+\sqrt 3\sinh\left(\frac{\lambda l\sqrt 3}{2}\right)\right)^{2}}
     $$
   Since the inequalities
     $$
     \cos\left(\frac{\lambda l}{2}-\varphi\right)<\cosh\frac{\lambda l\sqrt 3}{2},\quad
     \sin\left(\frac{\lambda l}{2}-\varphi\right)<\sinh\frac{\lambda l\sqrt 3}{2}
     $$
   hold for $\lambda>\lambda_{0}(\varphi)$, then
     $$
     f'(\lambda,\varphi)<l
       \frac{1-\cos^{2}\left(\frac{\lambda l}{2}-\varphi\right)-\sqrt 3\sin^{2}\left(\frac{\lambda l}{2}-\varphi\right)}
       {\left(\sin\left(\frac{\lambda l}{2}-\varphi\right)+\sqrt 3\sinh\left(\frac{\lambda l\sqrt 3}{2}\right)\right)^{2}}<0.
     $$
   Therefore, $f(\lambda,\varphi)$ monotonically decreases on the semi-axis $(\lambda_{0}(\varphi),\infty)$ and $f(\lambda,\varphi)\rightarrow -\frac{1}{\sqrt 3}$ as $\lambda\rightarrow\infty$. Then the equation (\ref{eq8.15}) for $\lambda>\lambda_{0}(\varphi)$ on each of the intervals
      $$
      \left(\frac{\pi}{l}(2n-1)-\varphi,\,\frac{\pi}{l}(2n+1)-\varphi\right)\quad (n\in\mathbb{N})
      $$
   has a single simple root
      $$
      \lambda_{n}(\varphi)=\frac{2\pi n}{l}-\varphi-\varepsilon_{n}\quad\left(0<\varepsilon_{n}<\frac{\pi}{l}\right).
      $$
   Asymptotics (\ref{eq8.13}) follows from (\ref{eq8.15}) and the relation
      $$
      f(\lambda,\varphi)\rightarrow -\frac{1}{\sqrt 3}\quad (\lambda\rightarrow\infty).
      $$

   Finally, note that $\Delta_{\theta}(\lambda)$ has no zeros outside the lines $l_{s}=\{z\in\mathbb{C}\mid z=x\zeta_{1}^{s}, x\in\mathbb{R}\}$ $(s=0,1,2)$. Indeed, if $w\not\in\bigcup l_{s}$ is zero, then the equation (\ref{eq8.9}) has a non-zero solution. Then there is an eigenfunction with complex eigenvalue $w^{3}\not\in \mathbb{R}$, which is impossible, since the operator $L_{\theta}$ is self-adjoint.
   \end{proof}

   The eigenfunctions of the operator $L_{\theta}$ have the form
     \begin{equation}\label{eq8.17}
     u(\mu_{n},x)=\frac{\theta s_{1}(i\mu_{n} x)s_{1}(i\mu_{n} l)+s_{2}(i\mu_{n} x)(1-\theta s_{0}(i\mu_{n} l))}{a_{n}(\mu_{n})},
     \end{equation}
   where $\mu_{n}\in\mathcal{Z}_{0}$ (see (\ref{eq8.12})) and the numbers $a_{n}(\mu_{n})$ are chosen from the normalization condition $\|u(\mu_{n},x)\|_{L^{2}(0,l)=1}$.
    \begin{corollary}\label{cor8.1}
    On the domain $D\subset L^{2}(0,l)$ of the operator $L_{\theta}$, the functions \emph{(\ref{eq8.17})} form an orthonormal basis.
    \end{corollary}

    \begin{remark}\label{rem8.3}
    The boundary conditions
       \begin{equation}\label{eq8.18}
       y(0)=ih_{0}y''(0),\; y'(0)=\theta y'(l),\; y(l)=ih_{l}y''(l),\;(h_{j}\in\mathbb{R}, \theta\in \mathbb{T})
       \end{equation}
    for the operation $(-iD)^{3}$ also generate a self-adjoint operator $L_{\theta,h_{0},h_{l}}$. Its eigenfunctions also form an orthonormal basis on the corresponding domain of the operator  $L_{\theta,h_{0},h_{l}}$.
    \end{remark}

    \section{\large{Operator resolvent}}\label{s:9}

    Calculate the resolvent $R_{L_{\theta}}(\lambda^{3})=(L_{\theta}-\lambda^{3}I)^{-1}$.

    Let $y=R_{L_{\theta}}(\lambda^{3})f$, i. e., $L_{\theta}y=\lambda^{3}y+f$. By Lemma \ref{lem6.1}, $y(\lambda,x)$ has the form
      \begin{equation}\label{eq9.1}
      y(\lambda,x)=y_{1}\frac{s_{1}(i\lambda x)}{i\lambda}+
       y_{2}\frac{s_{2}(i\lambda x)}{(i\lambda)^{2}}+
          i\int_{0}^{x}\frac{s_{2}(i\lambda(x-t))}{\lambda^{2}}f(t)dt
      \end{equation}
    and satisfies the boundary condition $y(\lambda,0)=0$. Since
      $$
      y'_{x}(\lambda,x)=y_{1}s_{0}(i\lambda x)+
       y_{2}\frac{s_{1}(i\lambda x)}{i\lambda}-
          \int_{0}^{x}\frac{s_{1}(i\lambda(x-t))}{\lambda}f(t)dt,
      $$
    then from the boundary conditions $y'(0)=\theta y'(l)$, $y(l)=0$ we obtain the system of equations
      $$
      \begin{cases}
      y_{1}(\theta s_{0}(i\lambda l)-1)+y_{2}\theta\frac{s_{1}(i\lambda l)}{i\lambda}=
          \theta\int_{0}^{l}\frac{s_{1}(i\lambda(l-t))}{\lambda}f(t)dt\\
          \\
      y_{1}\frac{s_{1}(i\lambda l)}{i\lambda}+y_{2}\frac{s_{2}(i\lambda l)}{(i\lambda)^{2}}=
             -i\int_{0}^{l}\frac{s_{2}(i\lambda(l-t))}{\lambda^{2}}f(t)dt.
      \end{cases}
      $$
    The determinant of this system coincides with the characteristic function $\Delta_{\theta}(\lambda)$ (see \ref{eq8.8}). Then
      $$
      y_{1}=\frac{\theta}{\lambda^{3}\Delta_{\theta}(\lambda)}\int_{0}^{l}
      \left[s_{2}(i\lambda(l-t))s_{1}(i\lambda l)-s_{1}(i\lambda(l-t))s_{2}(i\lambda l)\right]f(t)dt,
      $$
     \begin{multline}\nonumber
     y_{2}=\\
     \frac{i}{\lambda^{2}\Delta_{\theta}(\lambda)}\int_{0}^{l}
      \left[s_{2}(i\lambda(l-t))(1-\theta s_{0}(i\lambda l))+\theta s_{1}(i\lambda(l-t))s_{1}(i\lambda l)\right]f(t)dt.
     \end{multline}
    Then (\ref{eq9.1}) becomes
      \begin{multline}\label{eq9.2}
      y(\lambda,x)=
      \frac{i}{\lambda^{4}\Delta_{\theta}(\lambda)}\times\\
        \left\{
        \theta s_{1}(i\lambda x)\int_{0}^{l}\left[s_{1}(i\lambda(l-t))s_{2}(i\lambda l)-
        s_{2}(i\lambda(l-t))s_{1}(i\lambda l)\right]f(t)dt-\right.\\
      s_{2}(i\lambda l)\int_{0}^{l}\left[s_{2}(i\lambda(l-t))(1-\theta s_{0}(i\lambda l))+
      \theta s_{1}(i\lambda(l-t))s_{1}(i\lambda l)\right]f(t)dt+\\
      \left.\int_{0}^{x}s_{2}(i\lambda(x-t))\left[s_{2}(i\lambda l)+\theta s_{2}(-i\lambda l)\right]f(t)dt\right\}.
      \end{multline}
    To simplify this expression, we need the lemma
      \begin{lemma}\label{lem9.1}
      Let $s_{k}(z)$ $(k=0,1,2)$ be $3$-hyperbolic functions. Then
        \begin{multline}\label{eq9.3}
        s_{1}(x)s_{1}(l-t)s_{2}(l)- s_{1}(x)s_{2}(l-t)s_{1}(l)+\\
        s_{2}(x)s_{2}(l-t)s_{0}(l)-s_{2}(x)s_{1}(l-t)s_{1}(l)=\\
        s_{2}(-t)s_{2}(x-l)-s_{2}(x-t)s_{2}(-l)
        \end{multline}
      for any $x,l,t\in\mathbb{C}$.
      \end{lemma}

      \begin{proof}
      Using (\ref{eq4.6}), we transform the left side of (\ref{eq9.3}):
       \begin{multline}\nonumber
       A=s_{1}(x)s_{2}(l)\left[s_{0}(l)s_{1}(-t)+\underline{s_{1}(l)s_{0}(-t)}+
       s_{2}(l)s_{2}(-t)\right]-\\
         s_{1}(x)s_{1}(l)\left[\underline{s_{0}(l)s_{2}(-t)}+s_{1}(l)s_{1}(-t)+
         s_{2}(l)s_{0}(-t)\right]+\\
       s_{2}(x)s_{0}(l)\left[s_{0}(l)s_{2}(-t)+\underline{\underline{s_{1}(l)s_{1}(-t)}}
       +s_{2}(l)s_{0}(-t)\right]-\\
       s_{2}(x)s_{1}(l)\left[\underline{\underline{s_{0}(l)s_{1}(-t)}}+s_{1}(l)s_{0}(-t)
       +s_{2}(l)s_{2}(-t)\right]=
       \end{multline}
     \begin{multline}\nonumber
     s_{2}(-t)\left[s_{1}(x)\left(s_{2}(l)^{2}-s_{0}(l)s_{1}(l)\right)+
              s_{2}(x)\left(s_{0}(l)^{2}-s_{1}(l)s_{2}(l)\right)\right]+\\
        s_{1}(-t)s_{1}(x))\left(s_{0}(l)s_{2}(l)-s_{1}(l)^{2}\right)+
           s_{0}(-t)s_{2}(x))\left(s_{0}(l)s_{2}(l)-s_{1}(l)^{2}\right).
     \end{multline}
    Using (\ref{eq4.11}), we obtain
      \begin{multline}\nonumber
      A=s_{2}(-t)\left[s_{1}(x)s_{1}(-l)+s_{2}(x)s_{0}(-l)\right]-\\
         s_{1}(-t)s_{1}(x)s_{2}(-l)-s_{0}(-t)s_{2}(x)s_{2}(-l)=\\
          s_{2}(-t)\left[s_{1}(x)s_{1}(-l)+s_{2}(x)s_{0}(-l)\right]-\\
            s_{2}(-l)\left[s_{1}(-t)s_{1}(x)+s_{0}(-t)s_{2}(x)\right].
      \end{multline}
    Applying (\ref{eq4.6}) again, we get
      $$
      A=s_{2}(-t)s_{2}(x-l)-s_{2}(x-t)s_{2}(-l).
      $$
    \end{proof}

    Applying the identity (\ref{eq9.3}) to the integrand in (\ref{eq9.2}), we obtain
      \begin{multline}\label{eq9.4}
      y(\lambda,x)=\\
      \frac{i}{\lambda^{4}\Delta_{\theta}(\lambda)}
       \left\{\int_{0}^{l}\left[\theta s_{2}(i\lambda(x-l))s_{2}(-i\lambda t)-
        \theta s_{2}(i\lambda(x-t))s_{2}(-i\lambda l)-\right.\right.\\
         \qquad\qquad \left.s_{2}(i\lambda x)s_{2}(i\lambda (l-t))\right]f(t)dt+\\
             \left.\int_{0}^{x}s_{2}(i\lambda(x-t))[s_{2}(i\lambda l)+\theta s_{2}(-i\lambda l)]f(t)dt\right\}.\quad
      \end{multline}

    \begin{theorem}\label{the9.1}
    The resolvent of the operator $L_{\theta}$ has the form
      \begin{multline}\label{eq9.5}
      \left(R_{L_{\theta}}(\lambda^{3})f\right)(x)=
      \frac{i}{\lambda^{4}\Delta_{\theta}(\lambda)}\times\\
        \left\{ \int_{0}^{x}s_{2}(i\lambda(x-t))s_{2}(i\lambda l)f(t)dt\right.-
          \theta\int_{x}^{l}s_{2}(i\lambda(x-t))s_{2}(-i\lambda l)f(t)dt+\\
       \left.\int_{0}^{l}\left[\theta s_{2}(i\lambda(x-l))s_{2}(-i\lambda t)-
          s_{2}(i\lambda x)s_{2}(i\lambda (l-t))\right]f(t)dt\right\},
      \end{multline}
    where $\Delta_{\theta}(\lambda)$ is determined by the relation \emph{(\ref{eq8.10})}.
    \end{theorem}

    Formula (\ref{eq9.5}) allows one to find spectral projections onto eigensubspaces of $L_{\theta}$ as a residue of the resolvent at the zeros of the characteristic function $\Delta_{\theta}(\lambda)$.

  \end{document}